\documentclass[final]{siamltex}

\usepackage{graphicx,epstopdf} 
\usepackage[caption=false]{subfig} 
\usepackage{amsmath,color}

\newcommand{\ba}{\begin{array}}
\newcommand{\ea}{\end{array}}
\newcommand{\be}{\begin{equation}}
\newcommand{\ee}{\end{equation}}
\newcommand{\ben}{\begin{equation*}}
\newcommand{\een}{\end{equation*}}
\newcommand{\bd}{\begin{equation*}}
\newcommand{\ed}{\end{equation*}}
\newcommand{\bi}{\begin{itemize}}
\newcommand{\ei}{\end{itemize}}
\newcommand{\bn}{\begin{enumerate}}
\newcommand{\en}{\end{enumerate}}
\newcommand{\pa}{\partial}
\newcommand{\f}{\frac}
\newcommand{\ci}{\cite}

\newtheorem{prob}{Problem}
\newtheorem{rem}{Remark}

\title{An energy-based discontinuous Galerkin method for the wave equation with advection}

\author{Lu Zhang\thanks{Dept. of Mathematics, Southern Methodist University, Dallas TX 75275 ({\tt luzhang@smu.edu})} \and
Thomas Hagstrom\thanks{Dept. of Mathematics, Southern Methodist University, Dallas TX 75275 ({\tt thagstrom@smu.edu}).
The first two authors were supported in part by NSF Grant DMS-1418871. Any conclusions or recommendations expressed in
the paper are those of the authors and do not necessarily reflect the views of the NSF.} \and Daniel Appel\"{o}\thanks{Dept. of Applied Mathematics, University of Colorado, Boulder, Boulder CO 80309 ({\tt daniel.appelo@colorado.edu})}} 

\begin{document}

\maketitle

\begin{abstract}
An energy-based discontinuous Galerkin method for the advective wave equation is proposed
and analyzed. Energy-conserving or energy-dissipating methods follow from simple, mesh-independent choices
of the inter-element fluxes, and both subsonic and supersonic advection is allowed. Error estimates in the
energy norm are established, and numerical experiments on structured grids display optimal convergence in
the $L^2$ norm for upwind fluxes. The method generalizes earlier work on energy-based discontinuous
Galerkin methods for second order wave equations which was restricted to energy
forms written as a simple sum of kinetic and potential energy. 
\end{abstract}



\pagestyle{myheadings}
\thispagestyle{plain}
\markboth{L. ZHANG, T. HAGSTROM AND D. APPEL\"O}{ENERGY-BASED DG METHOD}

\section{Introduction}
Discontinuous Galerkin (DG) methods are by now well-estab-lished as a method of choice for
solving first order systems in Friedrichs form \cite{HesthavenWarburton08}. In particular, they
are robust, high-order, and geometrically flexible. 
In contrast, analogous
methods for second order hyperbolic equations are less well-developed. Although it is possible to
rewrite second order equations in first order form, there are disadvantages. The first order
systems may require significantly more variables and boundary conditions, and they are only
equivalent to the original forms for constrained data. Moreover, it is typical that
the basic wave equations arising in physical theories are expressed as action principles for
a Lagrangian, leading directly to second order equations, and it is unclear that they can always
be rewritten in Friedrichs form. 

In our view, a good target for a general formulation of DG methods are 
so-called regularly hyperbolic
partial differential equations \cite[Ch. 5]{CAction}, which arise as the Euler-Lagrange
equations associated to a Lagrangian, $L\left(\mathbf{x},t,\mathbf{u},\f {\pa \mathbf{u}}{\pa x_j},
\f {\pa \mathbf{u}}{\pa t} \right)$.
As a first step to generalizing the energy-based DG formulation of \ci{DGwave}, which
applied to a restricted class of Lagrangians of the form $L=\f {1}{2} \arrowvert \f {\pa \mathbf{u}}{\pa t}
\arrowvert^2 - U(\nabla \mathbf{u},\mathbf{u})$, 
we focus in this paper on the scalar wave equation with advection. Now $L$ is given by
\bd
\f {1}{2} \left( \f {\pa u}{\pa t} + \mathbf{w} \cdot \nabla u \right)^2 - \f {c^2}{2}
\arrowvert \nabla u \arrowvert^2,  
\ed
leading to the equation
\be
\left( \f {\pa}{\pa t} + \mathbf{w} \cdot \nabla \right)^2 u = c^2 \Delta u , \label{adwe}
\ee
and an associated energy density 
\bd
\mathcal{E} =  \f {1}{2} \left( \f {\pa u}{\pa t} + \mathbf{w} \cdot \nabla u \right)^2 
+ \f {c^2}{2} \arrowvert \nabla u \arrowvert^2.
\ed 

Besides being a simple example of a second order regularly hyperbolic partial differential
equation which cannot be directly treated by the method proposed in \ci{DGwave}, the
advective wave equation is a physically interesting model of sound propagation in a uniform
flow. Moreover, we believe our methods could be generalized to treat more general
models used in aeroacoustics. Lastly we note that both subsonic and supersonic background
flows are possible, leading to distinct formulations of upwind fluxes. 

We develop here a DG method for (\ref{adwe}) which:
\begin{itemize}
\item Guarantees energy stability based on simply defined upwind,
or central fluxes without mesh-dependent parameters,
\item Does not introduce extra fields beyond the two needed (i.e. $u$
and $\f {\pa u}{\pa t}$).
\end{itemize}
We develop
our formulation in Section \ref{formulation}, derive energy and error estimates in
Section \ref{energyerr}, and display some simple numerical examples in one and two
space dimensions in Section \ref{experiments}. Note that the analysis yields a suboptimal
convergence rate by $1$ for central fluxes and by $1/2$ for upwind fluxes. For problems
in one space dimension we prove optimal estimates in the upwind case, and observe optimal
convergence in $L^2$ for upwind fluxes in experiments on regular meshes. 

A wide variety of other DG methods have been proposed to solve second order wave
equations, but we contend that none of these formulations directly treat (\ref{adwe}) 
or meet our criteria. Local DG \ci{ChouShuXing2014}
and hybridizable DG \ci{GodunovHDG} methods introduce first order spatial derivatives,
already doubling the number of fields in three space dimensions.
Moreover, the method in \ci{GodunovHDG}, while providing
general upwind fluxes, assumes the first order system is in Friedrichs form.
Methods which don't introduce additional spatial derivatives include nonsymmetric and
symmetric interior penalty methods \ci{RiviereWheelerWave,GSSwave,AgutBartDiaz11}.
For these the penalty
parameters need to be mesh-dependent to guarantee stability.

\section{DG formulation}\label{formulation}
As in \ci{DGwave}, we introduce a second scalar variable to produce a system which is first order in time:
\bd
v = \frac{\partial{u}}{\partial{t}}+\mathbf{w}\cdot\nabla u ,
\ed
\begin{eqnarray}\label{first_order_system}
\left\{
\begin{array}{ll}
\frac{\partial u}{\partial t}+\mathbf{w}\cdot\nabla u-v=0,     \\
\frac{\partial{v}}{\partial t}+\mathbf{w}\cdot\nabla v-c^2\Delta u=0.
\end{array}
\right.
\end{eqnarray}
Now the energy form is
\bd
\mathcal{E}(u,v) = \f {1}{2} v^2 + \f {c^2}{2} \arrowvert \nabla u \arrowvert^2 ,
\ed
and we find the change of energy on an element $\Omega_j$ is given by boundary contributions,
\begin{eqnarray}\label{energy}
\frac{d}{dt}\int_{\Omega_j} \mathcal{E}(u,v) =\int_{\partial \Omega_j}c^2v\nabla u\cdot\mathbf{n}-\frac{1}{2}c^2|\nabla u|^2\mathbf{w}\cdot\mathbf{n}-\frac{1}{2}v^2\mathbf{w}\cdot \mathbf{n} ,
\end{eqnarray}
where $\mathbf{n}$ denotes the outward-pointing unit normal.

To discretize we require that the components of the approximations, $(u^h,v^h)$ to $(u,v)$, restricted to $\Omega_j$, be polynomials of degree $q$ and $s$ respectively, that is, elements of $\mathcal{P}^{(q,s,m)} \equiv (\Pi^q)^m \times (\Pi^s)^m$, with $m$ being the dimension. Now we seek approximations to the system
which satisfy a discrete energy identity analogous to (\ref{energy}).  Consider the discrete energy in $\Omega_j$,
\begin{eqnarray}\label{discrete_energy}
E_{j}^{h}(t)=\int_{\Omega_j}\frac{1}{2}(v^h)^2+\frac{1}{2} c^2 |\nabla u^h|^2,
\end{eqnarray}
and its time derivative,
\begin{eqnarray*}
\frac{dE_{j}^h}{dt}=\int_{\Omega_j} v^{h} \frac{\partial v^h}{\partial t}+c^2\nabla u^h \cdot\nabla \frac{\partial u^h}{\partial t}.
\end{eqnarray*}
To develop a weak form which is compatible with the discrete energy, choosing $\phi_u \in (\Pi^q)^{m}$ and $\phi_v \in (\Pi^s)^{m}$, we test the first equation of (\ref{first_order_system}) with $-c^2\Delta \phi_{u}$, the second equation of (\ref{first_order_system}) with $\phi_{v}$, and add flux terms which vanish for the continuous problem. This results in the following equations,

\begin{multline}\label{u_DG}
\int_{\Omega_j}-c^2\Delta\phi_{u}(\frac{\partial u^h}{\partial t}+\mathbf{w}\cdot\nabla u^h-v^h) = \\ \int_{\pa \Omega_j}-c^2\nabla\phi_u\cdot\mathbf{n}(\frac{\partial u^h}{\partial t}+\mathbf{w}\cdot\nabla u^h-v^*)
-c^2\nabla\phi_u\cdot(\nabla u^*-\nabla u^h)\mathbf{w}\cdot\mathbf{n},
\end{multline}

\begin{multline}\label{v_DG}
\int_{\Omega_j}\phi_{v}(\frac{\partial v^h}{\partial t}+\mathbf{w}\cdot\nabla v^h-c^2\Delta u^h)= \\ 
\int_{\pa \Omega_j}c^2 \phi_{v}(\nabla u^*-\nabla u^h)\cdot\mathbf{n} 
- (v^*-v^h)\phi_{v}\mathbf{w}\cdot\mathbf{n}.
\end{multline}
In what follows it is useful to note that an integration by parts in (\ref{u_DG}) and (\ref{v_DG}) yields the alternative form,
\begin{multline}\label{u_DG_use}
\int_{\Omega_j}c^2\nabla \phi_u\cdot\nabla(\frac{\partial u^h}{\partial t}+\mathbf{w}\cdot\nabla u^h-v^h)= \\
\int_{\pa \Omega_j}c^2(v^*-v^h)\nabla\phi_{u}\cdot\mathbf{n}
-c^2\nabla \phi_{u}\cdot(\nabla u^*-\nabla u^h)\mathbf{w}\cdot \mathbf{n},
\end{multline}
\begin{multline}\label{v_DG_use}
\int_{\Omega_j}\phi_{v}\frac{\partial v^h}{\partial t}+\phi_{v}\mathbf{w}\cdot\nabla v^h+c^2\nabla u^h\cdot\nabla \phi_{v}=\\ 
\int_{\pa \Omega_j}c^2 \phi_{v}\nabla u^*\cdot\mathbf{n}-(v^*-v^h)\phi_{v}\mathbf{w}\cdot\mathbf{n}.
\end{multline}
Lastly, we must supplement (\ref{u_DG_use}) with an equation to determine the mean value of $\f {\pa u^h}{\pa t}$.
Precisely, for an arbitrary constant $\tilde{\phi}_u$ we have,
\begin{eqnarray} \label{mean_u}
\int_{\Omega_j} \tilde{\phi}_u (\frac{\partial u^h}{\partial t}+\mathbf{w}\cdot\nabla u^h-v^h) =0.
\end{eqnarray}
Note that this equation does not change the energy.

Setting $\mathbf{\Phi} = (\phi_u, \phi_v, \tilde{\phi}_u)$, $\mathbf{U} = (u^h, v^h)$ we arrive at our final form:
\begin{multline*}
\mathcal{B}(\mathbf{\Phi}, \mathbf{U}) =
 \sum_j\int_{\Omega_j} \Big[ (c^2\nabla\phi_u\cdot\nabla+\tilde{\phi}_u)(\frac{\partial u^h}{\partial t}+\mathbf{w}\cdot\nabla u^h-v^h)+\phi_v\frac{\partial v^h}{\partial t}+\phi_v\mathbf{w}\cdot\nabla v^h  \\ 
  + c^2\nabla\phi_v\cdot\nabla u^h \Big] -\sum_j\int_{\partial\Omega_j} \Big[ c^2(v^*-v^h)\nabla\phi_u\cdot\mathbf{n}+c^2\phi_v\nabla u^*\cdot\mathbf{n}  \\
  -c^2\nabla\phi_u\cdot(\nabla u^*-\nabla u^h)\mathbf{w}\cdot\mathbf{n} - (v^*-v^h)\phi_v\mathbf{w}\cdot\mathbf{n} \Big].
\end{multline*}

Denote by $\mathcal{N}$ the space of arbitrary constants on an element, then we may state the semidiscrete problem as
\begin{prob} \label{problem}
Find $\mathbf{U} = (u^h,v^h) \in \mathcal{P}^{q,s,m}$ such that for all $ \mathbf{\Phi} \in \mathcal{P}^{q,s,m} \times \mathcal{N}$. 
\be \label{problem_a}
\mathcal{B}(\mathbf{\Phi},\mathbf{U}) = 0.
\ee
\end{prob}
We then have the following result.
\begin{theorem}
Let $\mathbf{U}(t)$ and the fluxes $v^*$, $\nabla u^*$ be given.
Then $\frac{d\mathbf{U}}{dt}$ is uniquely determined, and the energy identity
\begin{multline}\label{derivative_discrete_energy_use}
\frac{dE_j^h}{dt} =  \int_{\partial\Omega_j} \Big[ -\frac{1}{2}c^2|\nabla u^h|^2\mathbf{w}\cdot\mathbf{n}-\frac{1}{2}(v^h)^2\mathbf{w}\cdot\mathbf{n}+c^2(v^*-v^h)\nabla u^h\cdot\mathbf{n} \\
 -c^2\nabla u^h\cdot(\nabla u^*-\nabla u^h)\mathbf{w}\cdot\mathbf{n}+c^2{v^h}\nabla u^*\cdot\mathbf{n}-v^h(v^*-v^h)\mathbf{w}\cdot\mathbf{n} \Big],
\end{multline}
holds. 
\end{theorem}
\begin{proof}
The system on the element $\Omega_j$ is linear in the time derivatives, and the mass matrix of $\frac{dv^h}{dt}$ is nonsingular.  The number of linear equations for $\frac{du^h}{dt}$, which equals the number of independent equations in (\ref{u_DG_use}) plus the equation in (\ref{mean_u}), matches the dimensionality of $(\Pi^q)^m$. If the data $v^h$, $v^*$, $\nabla u^h$, $\nabla u^*$ vanishes in (\ref{u_DG_use}), we must have $\frac{d u^h}{dt} = 0$, and so the
linear system is invertible. By setting $\Phi = (\mathbf{U},0)$ in Problem \ref{problem}, we obtain (\ref{derivative_discrete_energy_use}) directly.
\end{proof}

\subsection{Fluxes}\label{fluxes}
To complete the problem specification we must prescribe the states $\nabla u^\ast, v^*$  both at inter-element and physical boundaries. 
Let ``$+$'' refer to traces of data from outside and ``$-$'' represent traces of data from inside. Moreover, we introduce the notation
\begin{align*}
 \{\{v\}\} &=\frac{1}{2}(v^+ + v^-),    & [[v]] &=v^+ \mathbf{n}^+ +v^- \mathbf{n}^-, \\ 
 \{\{\nabla u\}\}&=\frac{1}{2}(\nabla u^++\nabla u^-), &  [[\nabla{u}]] &=\nabla {u^+} \cdot \mathbf{n}^+ + \nabla {u^-} \cdot \mathbf{n}^- . 
\end{align*}
We firstly consider the inter-element boundaries. For definiteness label two elements sharing a boundary by $1$ and $2$.  Then their net contribution to the energy derivative is the boundary integral of
\begin{multline*}
J^h = -\frac{1}{2}c^2|\nabla u_{1}^h|^2\mathbf{w}\cdot\mathbf{n}_1-\frac{1}{2}(v_{1}^{h})^2\mathbf{w}\cdot \mathbf{n}_1+c^2(v^*-v_{1}^h)\nabla u_{1}^h\cdot\mathbf{n}_1\\
-c^2\nabla u_{1}^h\cdot(\nabla u^*-\nabla u_{1}^{h})\mathbf{w}\cdot\mathbf{n}_1+c^2v_{1}^h\nabla u^*\cdot\mathbf{n}_1-v_{1}^h (v^*-v_1)\mathbf{w}\cdot\mathbf{n}_1\\
-\frac{1}{2}c^2|\nabla u_{2}^h|^2\mathbf{w}\cdot\mathbf{n}_2-\frac{1}{2}(v_{2}^{h})^2\mathbf{w}\cdot \mathbf{n}_2+c^2(v^*-v_{2}^h)\nabla u_{2}^h\cdot\mathbf{n}_2\\
-c^2\nabla u_{2}^h\cdot(\nabla u^*-\nabla u_{2}^{h})\mathbf{w}\cdot\mathbf{n}_2+c^2v_{2}^h\nabla u^*\cdot\mathbf{n}_2-v_{2}^h (v^*-v_2)\mathbf{w}\cdot\mathbf{n}_2
\end{multline*}
There will be energy conservation if $J^h = 0$, and a typical example is given by the {\textit{central flux}},
\bd 
 v^{*} = \{\{v\}\},
\ \   \nabla u^{*} = \{\{\nabla u\}\}.
\ed
To define upwind fluxes, which will lead to $J^h<0$ in the presence of jumps, we first assume
$\arrowvert \mathbf{w} \cdot \mathbf{n} \arrowvert \leq c$,
and introduce a flux splitting determined by a parameter $\xi>0$ which has units of $c$,
\begin{eqnarray*}
v\nabla u\cdot\mathbf{n}=\frac{1}{4\xi}(v+\xi\nabla u\cdot\mathbf{n})^2-\frac{1}{4\xi}(v-\xi\nabla u\cdot\mathbf{n})^2=F^+ - F^- .
\end{eqnarray*}
Now choose the boundary states so that $F^+$ is computed using values from the outside of the element, and $F^-$ using the values from inside.  That is, we enforce the equation for $l=1,2$,
\bd 
v^*-\xi\nabla u^*\cdot\mathbf{n}_l=v_{l}^h-\xi\nabla u_{l}^h\cdot\mathbf{n}_l, \ \ \xi > 0.
\ed
Solving and additionally setting the tangential components of $\nabla u^{\ast}$ to be the average of the values from each side
we derive what we call the {\textit{Sommerfeld flux}}:
\bd
v^*=\{\{v^h\}\}-\frac{\xi}{2}[[\nabla u^h]],\ \ \ \ \nabla u^*=-\frac{1}{2\xi}[[v^h]]+\{\{\nabla u^h\}\}.
\ed
For this choice we find
\bd 
 J^h=-\Big(\frac{\xi c^2}{2}[[\nabla u^h]]^2+\frac{c^2}{2\xi}\Big|[[v^h]]\Big|^2-(\frac{c^2}{2\xi}+\frac{\xi }{2} )[[\nabla u^h]][[v^h]]\cdot\mathbf{w}\Big).
\ed
Denoting by the subscript ${\tau}$ the orthogonal projection of any vector onto the tangent space of the element boundary, we can rewrite
this formula 
\begin{eqnarray*}
J^h & = &  -\Big(\frac{\xi c^2}{2}[[\nabla u^h]]^2+\frac{c^2}{2\xi}\Big|[[v^h]]\Big|^2 
-(\frac{c^2}{2\xi}+\frac{\xi }{2} )[[\nabla u^h]](([[v^h]]\cdot\mathbf{n})(\mathbf{w}\cdot\mathbf{n})+([[v^h]]_{\tau}\cdot \mathbf{w}_{\tau}))\Big) \\
 & = & -\Big(\frac{\xi c^2}{2}[[\nabla u^h]]^2+\frac{c^2}{2\xi}\Big|[[v^h]]\Big|^2-(\frac{c^2}{2\xi}+\frac{\xi }{2} )[[\nabla u^h]]([[v^h]]\cdot\mathbf{n})(\mathbf{w}\cdot\mathbf{n})\Big).
\end{eqnarray*}
Moreover, we have
\bd 
[[\nabla u]]([[v]]\cdot\mathbf{n})\leq \frac{\alpha}{2}[[\nabla u]]^2+\frac{1}{2\alpha}\Big|[[v]]\Big|^2, \ \ \alpha > 0,
\ed
so that
\bd
J^h\leq -\frac{\xi c^2}{2}[[\nabla u^h]]^2-\frac{c^2}{2\xi}\Big|[[v^h]]\Big|^2\\
+(\frac{c^2}{2\xi}+\frac{\xi }{2} )|\mathbf{w}\cdot\mathbf{n}|\Big(\frac{\alpha}{2}[[\nabla u^h]]^2+\frac{1}{2\alpha}\Big|[[v^h]]\Big|^2\Big).
\ed
Now, if
\bd 
-\frac{\xi c^2}{2}+(\frac{c^2}{2\xi}+\frac{\xi}{2})|\mathbf{w}\cdot\mathbf{n}|\frac{\alpha}{2}\leq 0\ \ \text{and} \ \  -\frac{c^2}{2\xi}+(\frac{c^2}{2\xi}+\frac{\xi}{2})|\mathbf{w}\cdot\mathbf{n}|\frac{1}{2\alpha}\leq 0,
\ed
which requires
\be \label{bound_alpha}
\frac{(c^2+\xi^2)|\mathbf{w}\cdot\mathbf{n}|}{2c^2}\leq\alpha\leq\frac{2\xi^2c^2}{(c^2+\xi^2)|\mathbf{w}\cdot\mathbf{n}|}.
\ee
Then, our numerical energy will not grow. In the following, we are going to claim the existence of
$\alpha$ satisfying (\ref{bound_alpha}). Since
\bd 
\frac{2\xi^2c^2}{(c^2+\xi^2)|\mathbf{w}\cdot\mathbf{n}|}-\frac{(c^2+\xi^2)|\mathbf{w}\cdot\mathbf{n}|}{2c^2}=\frac{4\xi^2c^4-(c^2+\xi^2)^2|\mathbf{w}\cdot\mathbf{n}|^2}{2c^2|\mathbf{w}\cdot\mathbf{n}|(c^2+\xi^2)} ,
\ed
if $4\xi^2c^4-(c^2+\xi^2)^2|\mathbf{w}\cdot\mathbf{n}|^2\geq 0$, i.e,
\be \label{bound_w_1}
|\mathbf{w}\cdot\mathbf{n}|\leq\frac{2\xi c^2}{c^2+\xi^2},
\ee 
we conclude that (\ref{bound_alpha}) can be satisfied. Therefore, we will have a decreasing energy if $|\mathbf{w}\cdot\mathbf{n}| < \frac{2\xi c^2}{c^2+\xi^2}$ and an unchanged energy if $|\mathbf{w}\cdot\mathbf{n}| = \frac{2\xi c^2}{c^2+\xi^2}$. Particularly, if $\xi = c$, we will get a decreasing energy if $|\mathbf{w}\cdot\mathbf{n}|<c$ and an unchanged energy if $|\mathbf{w}\cdot\mathbf{n}| = c$.\\

A general parametrization of the flux is given by
\bd
v^{*} = (\sigma v_{1}^{h} + (1-\sigma) v_{2}^{h}) - \eta[[\nabla u^{h}]],\ \ \
\nabla u^{*} = -\beta [[v^{h}]] + \big((1-\sigma)\nabla u_{1}^{h} + \sigma \nabla u_{2}^{h}\big),
\ed
with $0\leq\sigma\leq 1$, $\beta$,$\eta \geq 0$. For this general flux form, we find
\bd
J^{h} = -\Big(c^{2}\eta[[\nabla u^h]]^2+c^{2}\beta\Big|[[v^h]]\Big|^2-(c^{2}\beta+\eta )[[\nabla u^h]]([[v^h]]\cdot\mathbf{n})(\mathbf{w}\cdot\mathbf{n})\Big),
\ed
The previous situations correspond to the following:\\
\textit{Central flux}: $\sigma = \frac{1}{2}$, $\beta = \eta = 0$.\\
\textit{Sommerfeld flux}: $\sigma = \frac{1}{2}$, $\beta = \frac{1}{2\xi}$, $\eta = \frac{\xi}{2}$.\\

We also consider the situation $c < \arrowvert \mathbf{w} \cdot \mathbf{n} \arrowvert$ in which case
\textit{upwind fluxes} only come from one element. For example
\be \label{flux_onedir1}
 v^* = v_{1}^h,  \ \  \nabla u^* = \nabla u_{1}^h,
\ee
or
\be   \label{flux_onedir2}
v^{*} = v_{2}^h,\ \  \nabla u^{*} = \nabla u_{2}^h.
\ee
Based on (\ref{flux_onedir1}), we then have 
\begin{multline} \label{J_flux_onedir1}
J^{h} = \frac{1}{2}\Big(c^{2}\arrowvert \nabla u_{1}^h - \nabla u_{2}^h \arrowvert^{2} + (v_{1}^h - v_{2}^h)^{2}\Big)\mathbf{w}\cdot\mathbf{n}_{2}+c^{2}(\nabla u_{1}^h-\nabla u_{2}^h)\cdot\mathbf{n}_{1}(v_{1}^h-v_{2}^h) \\
 \leq \frac{(\mathbf{w}\cdot\mathbf{n}_{2}+c)}{2}\Big( c^2 \arrowvert \nabla u_{1}^h-\nabla u_{2}^h \arrowvert^{2}+(v_{1}^h-v_{2}^h)^{2}\Big).
\end{multline}
then $J^h \leq 0$ if $\mathbf{w}\cdot\mathbf{n}_2\leq -c$. From (\ref{flux_onedir2}), we have
\begin{multline} \label{J_flux_onedir2}
J^{h}  =  \frac{1}{2}\Big(c^{2} \arrowvert \nabla u_{1}^h - \nabla u_{2}^h \arrowvert^{2} + (v_{1}^h - v_{2}^h)^{2}\Big)\mathbf{w}\cdot\mathbf{n}_{1}+c^{2}(\nabla u_{1}^h-\nabla u_{2}^h)\cdot\mathbf{n}_{2}(v_{1}^h-v_{2}^h) \\
 \leq \frac{(\mathbf{w}\cdot\mathbf{n}_{1}+c)}{2}\Big(c^2 \arrowvert \nabla u_{1}^h-\nabla u_{2}^h \arrowvert^{2}+(v_{1}^h-v_{2}^h)^{2}\Big).
\end{multline}
then $J^h \leq 0$ if $\mathbf{w}\cdot\mathbf{n}_1\leq -c$.

\subsection{Boundary conditions}
Next we consider physical boundary conditions. We consider separately
inflow boundaries for which we have ${\bf w}\cdot{\bf n} < 0$ and outflow boundaries, where ${\bf w}\cdot{\bf n} > 0$.

\subsubsection{Inflow boundary conditions}
On an inflow boundary, $\mathbf{w}\cdot\mathbf{n}<0$, we choose the Dirichlet boundary condition,
$u(\mathbf{x},t) = 0$, which implies $\frac{\partial u(\mathbf{x},t)}{\partial t} = 0$. Then 
\[ 
 v(\mathbf{x},t) = \frac{\partial u(\mathbf{x},t)}{\partial t}+\mathbf{w}\cdot\nabla u(\mathbf{x},t)=\mathbf{w}\cdot\nabla u(\mathbf{x},t).
 \]
Considering the flux terms, assuming again that $c \geq \arrowvert \mathbf{w} \cdot \mathbf{n} \arrowvert$, we enforce the following conditions,
\begin{equation*} 
\left\{
\begin{array}{ll}
v^* - \mathbf{w}\cdot\nabla u^* = 0 ,     \\
v^* - \xi\nabla u^*\cdot\mathbf{n} = v^h - \xi\nabla u^h\cdot\mathbf{n}  ,\\
(\nabla u^*)_\tau = 0 .
\end{array}
\right.
\end{equation*}
Solving this system we find
\be  \label{inflow_fluxes}
\nabla u^*\cdot\mathbf{n} = \frac{\xi\nabla u^h\cdot\mathbf{n}-v^h}{\xi-\mathbf{w}\cdot\mathbf{n}}, \ \ v^* = \frac{\mathbf{w}\cdot\mathbf{n}}{\xi-\mathbf{w}\cdot\mathbf{n}}(\xi\nabla u^h\cdot\mathbf{n}-v^h).
\ee
where we have used the fact $\mathbf{w}\cdot\nabla u^* = (\mathbf{w}\cdot\mathbf{n})(\nabla u^*\cdot\mathbf{n}) + \mathbf{w}_{\tau} \cdot (\nabla u^*)_{\tau}$. Then through simple calculation we find 
\begin{multline} \label{equality_2}
-c^2\nabla u^h \cdot\nabla u^*-v^h v^*= \\
-\frac{c^2\xi}{\xi-\mathbf{w}\cdot\mathbf{n}}(\nabla u^h\cdot\mathbf{n})^2+\frac{c^2-\xi\mathbf{w}\cdot\mathbf{n}}{\xi-\mathbf{w}\cdot\mathbf{n}}(\nabla u^h\cdot\mathbf{n})v^h+\frac{\mathbf{w}\cdot\mathbf{n}}{\xi-\mathbf{w}\cdot\mathbf{n}}(v^h)^2.
\end{multline}
and
\be \label{equality_3}
c^2(v^*-v^h)\nabla u^h\cdot\mathbf{n}+c^2v^h\nabla u^*\cdot\mathbf{n}=\frac{c^2\xi\mathbf{w}\cdot\mathbf{n}}{\xi-\mathbf{w}\cdot\mathbf{n}}(\nabla u^h\cdot\mathbf{n})^2-\frac{c^2}{\xi-\mathbf{w}\cdot\mathbf{n}}(v^h)^2.
\ee
Denote by the subscript $I$ faces with inflow. Plugging (\ref{inflow_fluxes}) into (\ref{derivative_discrete_energy_use}) and using (\ref{equality_2}) and (\ref{equality_3}) to simplify the resulting equation we find 
\begin{multline*} 
\frac{dE_{j_I}^h}{dt} = \int_{\partial \Omega_{j_I}}\Big(\frac{1}{2}c^2|\nabla u^h|^2+\frac{1}{2}(v^h)^2-c^2\nabla u^h\cdot\nabla u^*-v^hv^*\Big)\mathbf{w}\cdot\mathbf{n} \nonumber \\  +c^2(v^*-v^h)\nabla u^h\cdot\mathbf{n} 
 +c^2{v^h}\nabla u^*\cdot\mathbf{n} \\ 
 =\int_{\partial\Omega_{j_I}}\frac{c^2\mathbf{w}\cdot\mathbf{n}}{2}\arrowvert \nabla u^h_{\tau} \arrowvert^2 + \frac{c^2\mathbf{w}\cdot\mathbf{n}}{2}(\nabla u^h\cdot\mathbf{n})^2 +     
\Big(\frac{\mathbf{w}\cdot\mathbf{n}}{2} 
+\frac{(\mathbf{w}\cdot\mathbf{n})^2-c^2}{\xi-\mathbf{w}\cdot\mathbf{n}}\Big) 
 \cdot (v^h)^2 \\
 + \frac{(c^2-\xi\mathbf{w}\cdot\mathbf{n})\mathbf{w}\cdot\mathbf{n}}{\xi-\mathbf{w}\cdot\mathbf{n}}(\nabla u^h\cdot\mathbf{n})v^h.
\end{multline*}
Noting that $\mathbf{w}\cdot\mathbf{n}<0$ on the inflow boundaries, denote $a = \frac{c^2\mathbf{w}\cdot\mathbf{n}}{2}$, $b = \frac{(c^2-\xi\mathbf{w}\cdot\mathbf{n})\mathbf{w}\cdot\mathbf{n}}{2(\xi-\mathbf{w}\cdot\mathbf{n})}$ and $d = \frac{\mathbf{w}\cdot\mathbf{n}}{2}+\frac{(\mathbf{w}\cdot\mathbf{n})^2-c^2}{\xi-\mathbf{w}\cdot\mathbf{n}}$, then we will get a decreasing energy if $ad > b^2$.
 Now, let us claim the fact $ad > b^2$. Since
 \bd 
 ad  = \frac{c^2\xi^2(\mathbf{w}\cdot\mathbf{n})^2+2c^4(\mathbf{w}\cdot\mathbf{n})^2-2c^4\xi(\mathbf{w}\cdot\mathbf{n})-c^2(\mathbf{w}\cdot\mathbf{n})^4}{4(\xi-\mathbf{w}\cdot\mathbf{n})^2},
 \ed
 and
 \bd 
 b^2  = \frac{c^4(\mathbf{w}\cdot\mathbf{n})^2+\xi^2(\mathbf{w}\cdot\mathbf{n})^4-2c^2\xi(\mathbf{w}\cdot\mathbf{n})^3}{4(\xi-\mathbf{w}\cdot\mathbf{n})^2},
 \ed
by simple calculation, we find that the numerator of $ad-b^2$ is
\bd
(c^2-(\mathbf{w}\cdot\mathbf{n})^2)(\mathbf{w}\cdot\mathbf{n})((c^2+\xi^2)\mathbf{w}\cdot\mathbf{n}-2c^2\xi),
\ed
 then if $-c<\mathbf{w}\cdot\mathbf{n}<0$, we will have $ad > b^2$. Thus, we will get a decreasing energy if $-c<\mathbf{w}\cdot\mathbf{n}<0$ and an unchanged energy if $\mathbf{w}\cdot\mathbf{n} = -c$ and $(\nabla u^h)_{\tau}=0$ on the inflow boundaries.

We remark that a straightforward derivation of this estimate leads to constants which are unbounded as $\mathbf{w}
\cdot \mathbf{n} \rightarrow 0$, but using the fact that $v^{\ast} \rightarrow 0$ also the bounds can be maintained. 

If $\mathbf{w} \cdot \mathbf{n} < -c$ we must impose two boundary conditions, $u=0$, $\nabla u \cdot \mathbf{n}=0$
from which we deduce $v^{\ast}=0$, $\nabla u^{\ast}=\mathbf{0}$. Then from (\ref{J_flux_onedir2}) we conclude that the
energy is decreasing.

\subsubsection{Outflow boundary conditions}
 At outflow boundaries, $0 < \mathbf{w}\cdot\mathbf{n} \leq c$, we impose the radiation boundary conditions,
\bd 
\left\{
\begin{array}{ll}
v^* + \xi\nabla u^*\cdot\mathbf{n} = 0 ,    \\
v^* - \xi\nabla u^*\cdot\mathbf{n} = v^h - \xi\nabla u^h\cdot\mathbf{n}  ,\\
(\nabla u^*)_\tau = (\nabla u^h)_\tau.
\end{array}
\right.
\ed
Solving this system, we find that
\be  \label{outflow_fluxes}
\nabla u^*\cdot\mathbf{n} = \frac{\xi\nabla u^h\cdot\mathbf{n}-v^h}{2\xi},\ \  v^* = \frac{v^h-\xi\nabla u^h\cdot\mathbf{n}}{2}.
\ee 
By a simple calculation we obtain
\begin{multline} \label{equality_7}
-c^2\nabla u^h\cdot\nabla u^* - v^h v^*\\
=-\frac{c^2}{2}(\nabla u^h\cdot\mathbf{n})^2-\frac{1}{2}(v^h)^2+(\frac{c^2}{2\xi}+\frac{\xi}{2})(\nabla u^h\cdot\mathbf{n})v^h-c^2 \arrowvert (\nabla u^h)_{\tau}\arrowvert^2,
\end{multline}
and
\be  \label{equality_8}
c^2(v^*-v^h)\nabla u^h\cdot\mathbf{n} + c^2v^h\nabla u^*\cdot\mathbf{n} =-\frac{c^2\xi}{2}(\nabla u^h\cdot\mathbf{n})^2 -\frac{c^2}{2\xi}(v^h)^2.
\ee
Denote by the subscript $O$ faces that have outflow. Using (\ref{outflow_fluxes}) in (\ref{derivative_discrete_energy_use}) and applying (\ref{equality_7}) and (\ref{equality_8}) to the resulting equation, we conclude that
\begin{multline*} 
\frac{dE_{j_O}^h}{dt} = \int_{\partial \Omega_{j_O}} \Bigg[ \Big(\frac{1}{2}c^2|\nabla u^h|^2+\frac{1}{2}(v^h)^2-c^2\nabla u^h\cdot\nabla u^*-v^hv^*\Big)\mathbf{w}\cdot\mathbf{n} \\
+c^2(v^*-v^h)\nabla u^h\cdot\mathbf{n}+c^2{v^h}\nabla u^*\cdot\mathbf{n} \Bigg]\\
=\int_{\partial\Omega_{j_O}} \Big[-c^2 \arrowvert (\nabla u^h)_{\tau}\arrowvert^2\mathbf{w}\cdot\mathbf{n}+\left(\frac{c^2}{2\xi}+\frac{\xi}{2}\right)(\nabla u^h\cdot\mathbf{n})v^h  \mathbf{w}\cdot\mathbf{n}
-\frac{c^2\xi}{2}(\nabla u^h\cdot\mathbf{n})^2 -\frac{c^2}{2\xi}(v^h)^2\Big].
\end{multline*}
For positive $\delta$, we have that
\begin{multline*} 
\Big(\frac{c^2}{2\xi}+\frac{\xi}{2}\Big)(\nabla u^h\cdot\mathbf{n})v^h\mathbf{w}\cdot\mathbf{n}
 \leq \frac{\delta}{2}\Big(\frac{c^2}{2\xi}+\frac{\xi}{2}\Big)\mathbf{w}\cdot\mathbf{n}(\nabla u^h\cdot\mathbf{n})^2+\frac{1}{2\delta}\Big(\frac{c^2}{2\xi}+\frac{\xi}{2}\Big)\mathbf{w}\cdot\mathbf{n}(v^h)^2,
\end{multline*}
then we get a decreasing energy if the following conditions are satisfied
\bd 
\frac{\delta}{2}\Big(\frac{c^2}{2\xi} + \frac{\xi}{2}\Big)\mathbf{w}\cdot\mathbf{n}-\frac{c^2\xi}{2}\leq 0,\ \ \
\frac{1}{2\delta}\Big(\frac{c^2}{2\xi}+\frac{\xi}{2}\Big)\mathbf{w}\cdot\mathbf{n}
-\frac{c^2}{2\xi}\leq 0.
\ed
This is equivalent to 
\bd 
\frac{(c^2+\xi^2)\mathbf{w}\cdot\mathbf{n}}{2c^2}\leq\delta\leq\frac{2c^2\xi^2}{(c^2+\xi^2)\mathbf{w}\cdot\mathbf{n}}.
\ed
Now, the existence of $\delta$ follows as
\bd 
\frac{2c^2\xi^2}{(c^2+\xi^2)\mathbf{w}\cdot\mathbf{n}}-\frac{(c^2+\xi^2)\mathbf{w}\cdot\mathbf{n}}{2c^2} = \frac{4c^4\xi^2-(c^2+\xi^2)^2(\mathbf{w}\cdot\mathbf{n})^2}{2c^2(c^2+\xi^2)\mathbf{w}\cdot\mathbf{n}},
\ed
and then $\delta$ exists if $4c^4\xi^2-(c^2+\xi^2)^2(\mathbf{w}\cdot\mathbf{n})^2\geq0$, i.e. $(\mathbf{w}\cdot\mathbf{n})^{2}\leq\frac{4c^{4}\xi^{2}}{(c^{2}+\xi^{2})^2}$ which in turn gives us a decreasing energy if $0 < \mathbf{w}\cdot\mathbf{n} < \frac{2\xi c^2}{c^2+\xi^2}$, an unchanged energy if $\mathbf{w}\cdot\mathbf{n} = \frac{2\xi c^2}{c^2+\xi^2}$ and $(\nabla u^h)_{\tau} = 0$ on the outflow boundaries. Moreover, if we choose $\xi = c$, we will have a decreasing energy when $0 < \mathbf{w}\cdot\mathbf{n} < c$, an unchanged energy when $\mathbf{w}\cdot\mathbf{n} = c$ and $(\nabla u^h)_{\tau} = 0$ on the outflow boundaries.

Lastly we note that if $\mathbf{w \cdot n} > c$ we impose no boundary conditions, set $v^{\ast}=v^h$ and
$\nabla u^{\ast} = \nabla u^h$, and invoke (\ref{J_flux_onedir1}) to conclude that the energy is decreasing. 

Combining all the inter-element and physical boundaries we can now state the  final energy identity.
\begin{theorem}
The discrete energy $E^h(t) = \sum_j E_j^h(t)$ with $E_j^h(t)$ defined in (\ref{discrete_energy}) satisfies
\begin{multline} \label{derivative_discrete_energy_all}
\frac{dE^h}{dt} = -\sum_j\int_{F_j} 
\Big[ 
c^{2}\eta[[\nabla u^h]]^2+c^{2}\beta\big|[[v^h]]\big|^2
-c^{2}\beta[[v^h]]\cdot\big(\nabla u^h_1(\mathbf{w}\cdot\mathbf{n}_1)
+\nabla u^h_2(\mathbf{w}\cdot\mathbf{n}_2) \big)  \\
- \eta [[\nabla u^h]] [[v^h]] \cdot \mathbf{w} \Big] 
 +\sum_j\int_{B_{j_I}} \Big[ \frac{c^2\mathbf{w}\cdot\mathbf{n}}{2} \arrowvert (\nabla u^h)_{\tau}\arrowvert^2 
 +  \frac{c^2\mathbf{w}\cdot\mathbf{n}}{2} (\nabla u^h\cdot\mathbf{n})^2 \\
+  \Big(\frac{\mathbf{w}\cdot\mathbf{n}}{2}
+\frac{(\mathbf{w}\cdot\mathbf{n})^2-c^2}{\xi-\mathbf{w}\cdot\mathbf{n}}\Big)(v^h)^2 
+ \frac{(c^2-\xi\mathbf{w}\cdot\mathbf{n})\mathbf{w}\cdot\mathbf{n}}{\xi-\mathbf{w}\cdot\mathbf{n}}(\nabla u^h\cdot\mathbf{n})v^h \Big]  \\ 
+ \sum_j\int_{B_{j_O}} \Big[ 
-\frac{c^2}{2} \arrowvert (\nabla u^h)_{\tau}\arrowvert^2\mathbf{w}\cdot\mathbf{n}
+(\frac{c^2}{2\xi}
+\frac{\xi}{2})(\nabla u^h\cdot\mathbf{n})v^h\mathbf{w}\cdot\mathbf{n}\\
-\frac{c^2\xi}{2}(\nabla u^h\cdot\mathbf{n})^2-\frac{c^2}{2\xi}(v^h)^2 \Big],
\end{multline}
where $F_j$ represents the inter-element boundaries, $B_{jI}$ represents inflow physical boundaries and $B_{jO}$ represents outflow boundaries.
\end{theorem}

If the flux parameters $\sigma$, $\beta$ and $\eta$ are chosen based on Section \ref{fluxes}, then $\frac{dE^h(t)}{dt} \leq 0$.

\section{Error estimates in the energy norm}\label{energyerr}
We define the errors by
\bd 
e_u = u - u^h, \ \  e_v = v - v^h,
\ed
and let
\bd 
\mathbf{D}^h = (e_u, e_v).
\ed
Note the fundamental Galerkin orthogonality relation:
\bd 
\mathcal{B}(\mathbf{\Phi}, \mathbf{D}^h) = 0.
\ed

To proceed we follow the standard approach of comparing $(u^h, v^h)$ to an arbitrary polynomial $(\tilde{u}^h, \tilde{v}^h)\in\mathcal{P}^{q,s,m}$. we define the differences
\bd
\tilde{e}_u = \tilde{u}^h - u^h, 
\ \ \tilde{e}_v = \tilde{v}^h - v^h, 
\ \  \delta_u = \tilde{u}^h - u, 
\ \  \delta_v = \tilde{v}^h - v,
\ed
and let
\bd
\tilde{\mathbf{D}}^h = (\tilde{e}_u, \tilde{e}_v)\in\mathcal{P}^{q,s,m}, 
\ \  \tilde{\mathbf{D}}_0^h = (\tilde{e}_u, \tilde{e}_v, 0)\in\mathcal{P}^{q,s,m}\times\mathcal{N}, 
\ \  \mathbf{\Delta}^h = (\delta_u, \delta_v).
\ed
Then, since $\mathbf{D}^h = \tilde{\mathbf{D}}^h - \mathbf{\Delta}^h$, we have the error equation
\bd 
\mathcal{B}(\tilde{\mathbf{D}}_0^h, \tilde{\mathbf{D}}^h)=\mathcal{B}(\tilde{\mathbf{D}}_0^h, \mathbf{\Delta}^h).
\ed
Finally, define the energy of $\tilde{\mathbf{D}}^h$ by
\bd 
\mathcal{E}^h = \frac{1}{2}\sum_{j}\int_{\Omega_j}\tilde{e}_v^2 + c^2|\nabla\tilde{e}_u|^2.
\ed
Then repeating the arguments which led to (\ref{derivative_discrete_energy_all}), we derive:
\begin{multline} \label{derivative_error_energy}
\frac{d\mathcal{E}^h}{dt} = 
\mathcal{B}(\tilde{D}_0^h, \mathbf{\Delta}^h) \\
- \sum_j\int_{F_j} \Big[ c^{2}\eta[[\nabla \tilde{e}_u]]^2
 + c^{2}\beta\big|[[\tilde{e}_v]]\big|^2
  - c^{2}\beta[[\tilde{e}_v]]\cdot\big(\nabla\tilde{e}_{u1}(\mathbf{w}\cdot\mathbf{n}_1)  \\
 +\nabla \tilde{e}_{u2}(\mathbf{w}\cdot\mathbf{n}_2)\big)
  -\eta [[\nabla \tilde{e}_u]][[\tilde{e}_v]]\cdot\mathbf{w} \Big] \\
+ \sum_j\int_{B_{j_I}} 
\Big[\frac{c^2\mathbf{w}\cdot\mathbf{n}}{2} \arrowvert (\nabla\tilde{e}_u)_{\tau} \arrowvert^2 
 + \frac{c^2\mathbf{w}\cdot\mathbf{n}}{2}(\nabla\tilde{e}_u\cdot\mathbf{n})^2 
+ \Big(\frac{\mathbf{w}\cdot\mathbf{n}}{2}
+\frac{(\mathbf{w}\cdot\mathbf{n})^2-c^2}{\xi-\mathbf{w}\cdot\mathbf{n}}\Big)(\tilde{e}_v)^2 \\
 + \frac{(c^2-\xi\mathbf{w}\cdot\mathbf{n})\mathbf{w}\cdot\mathbf{n}}{\xi-\mathbf{w}\cdot\mathbf{n}}(\nabla \tilde{e}_u\cdot\mathbf{n})\tilde{e}_v \Big] 
+ \sum_j\int_{B_{j_O}} \Big[ -\frac{c^2}{2} \arrowvert (\nabla \tilde{e}_u)_{\tau}\arrowvert^2\mathbf{w}\cdot\mathbf{n}   \\
  + (\frac{c^2}{2\xi}+\frac{\xi}{2})(\nabla \tilde{e}_u\cdot\mathbf{n})\tilde{e}_v\mathbf{w}\cdot\mathbf{n}
 -\frac{c^2\xi}{2}(\nabla \tilde{e}_u\cdot\mathbf{n})^2-\frac{c^2}{2\xi}(\tilde{e}_v)^2 \Big].  
\end{multline}

We now must choose $(\tilde{u}^h, \tilde{v}^h)$ to achieve an acceptable error estimate. In what follows we will assume for simplicity that $(u^h, v^h) = (\tilde{u}^h, \tilde{v}^h)$ at $t = 0$. We note that in the numerical experiments we found it beneficial to subtract off a function satisfying the initial conditions, thus solving a forced equation with zero initial data. On $\Omega_j$ we impose for all times $t$ and $(\phi_u, \phi_v)\in\mathcal{P}^{q,s,m}$,
\be \label{restriction_polynomials}
\int_{\Omega_j}\nabla\phi_u\cdot\nabla\delta_u = \int_{\Omega_j}\phi_v\delta_v = \int_{\Omega_j}\delta_u = 0 .
\ee
Then, integrating by parts, we derive the following expression for $\mathcal{B}(\tilde{\mathbf{D}}_0^h, \mathbf{\Delta}^h)$: 
\begin{multline*} 
\mathcal{B}(\tilde{\mathbf{D}}_0^h, \mathbf{\Delta}^h) = \sum_j\int_{\Omega_j} \Big[ c^2\nabla\tilde{e}_u\cdot\nabla(\frac{\partial \delta_u}{\partial t}) 
 -c^2\nabla^2\tilde{e}_u\mathbf{w}\cdot\nabla\delta_u 
  + c^2\nabla^2\tilde{e}_u\delta_v + \tilde{e}_v(\frac{\partial \delta_v}{\partial t})  \\
- \mathbf{w}\cdot\nabla\tilde{e}_v\delta_v 
+ c^2\nabla\tilde{e}_v\cdot\nabla\delta_u \Big]
- \sum_j \int_{\partial\Omega_j} \Big[ -c^2\nabla\tilde{e}_u\cdot\mathbf{n}\nabla\delta_u\cdot\mathbf{w}  
+ c^2\delta_v^*\nabla\tilde{e}_u\cdot\mathbf{n} \\
+ c^2\tilde{e}_v\nabla\delta_u^*\cdot\mathbf{n} 
- c^2\nabla\tilde{e}_u\cdot(\nabla\delta^*_u-\nabla\delta_u)\mathbf{w}\cdot\mathbf{n}
- \delta_v^*\tilde{e}_v\mathbf{w}\cdot\mathbf{n} \Big],
\end{multline*}
now  we can rewrite the volume integral $\int_{\Omega_j}\nabla^2\tilde{e}_u\mathbf{w}\cdot\nabla\delta_u$ as,
\bd 
\int_{\Omega_j}\nabla^2\tilde{e}_u\mathbf{w}\cdot\nabla\delta_u = \int_{\Omega_j}\big(\nabla(\mathbf{w}\cdot\nabla\tilde{e}_u)+\nabla\times(\mathbf{w}\times\nabla\tilde{e}_u)\big)\cdot\nabla\delta_u,
\ed
then invoking (\ref{restriction_polynomials}) the volume integrals in $\mathcal{B}(\tilde{\mathbf{D}}_0^h, \mathbf{\Delta}^h)$ will vanish and we can simplify $\mathcal{B}(\tilde{\mathbf{D}}_0^h, \mathbf{\Delta}^h)$ to
\begin{multline} \label{B_definition_2}
\mathcal{B}(\tilde{\mathbf{D}}_0^h, \mathbf{\Delta}^h) =
- \sum_j \int_{\partial\Omega_j}c^2\big((\mathbf{w}\times\nabla\tilde{e}_u)\times\nabla\delta_u\big)\cdot\mathbf{n}
-c^2\nabla\tilde{e}_u\cdot\mathbf{n}\nabla\delta_u\cdot\mathbf{w}\\
+ c^2\delta_v^*\nabla\tilde{e}_u\cdot\mathbf{n} + c^2\tilde{e}_v\nabla\delta_u^*\cdot\mathbf{n} 
-c^2\nabla\tilde{e}_u\cdot(\nabla\delta^*_u-\nabla\delta_u)\mathbf{w}\cdot\mathbf{n}
- \delta_v^*\tilde{e}_v\mathbf{w}\cdot\mathbf{n}\\
=- \sum_j \int_{\partial\Omega_j}c^2\big(-(\nabla\delta_u\cdot\nabla\tilde{e}_u)\mathbf{w}\cdot\mathbf{n}
+\nabla\tilde{e}_u\cdot\mathbf{n}\nabla\delta_u\cdot\mathbf{w}\big)
-c^2\nabla\tilde{e}_u\cdot\mathbf{n}\nabla\delta_u\cdot\mathbf{w} \\
+ c^2\delta_v^*\nabla\tilde{e}_u\cdot\mathbf{n} 
+ c^2\tilde{e}_v\nabla\delta_u^*\cdot\mathbf{n} 
-c^2\nabla\tilde{e}_u\cdot(\nabla\delta^*_u-\nabla\delta_u)\mathbf{w}\cdot\mathbf{n}- \delta_v^*\tilde{e}_v\mathbf{w}\cdot\mathbf{n}\\
=- \sum_j \int_{\partial\Omega_j}c^2\delta_v^*\nabla\tilde{e}_u\cdot\mathbf{n} - \delta_v^* \tilde{e}_v\mathbf{w}\cdot\mathbf{n}+ c^2\tilde{e}_v\nabla\delta_u^*\cdot\mathbf{n} - c^2\nabla\tilde{e}_u\cdot\nabla\delta_u^*\mathbf{w}\cdot\mathbf{n} .  
\end{multline}

Combining contributions from neighboring elements we then have 
\begin{multline*} 
\mathcal{B}(\tilde{\mathbf{D}}_0^h, \mathbf{\Delta}^h) 
= - \sum_j \int_{F_j} \Big[ c^2[[\nabla\tilde{e}_u]]\delta_v^* 
-[[\tilde{e}_v]]\delta_v^*\cdot\mathbf{w}
+ c^2[[\tilde{e}_v]]\cdot\nabla\delta_u^*\\
- c^2\nabla\tilde{e}_{u1}\cdot\nabla\delta_u^*\mathbf{w}\cdot\mathbf{n}_1
 - c^2\nabla\tilde{e}_{u2}\cdot\nabla\delta_u^*\mathbf{w}\cdot\mathbf{n}_2 \Big]\\
  - \sum_j \int_{B_j} \Big[ c^2\delta_v^*\nabla\tilde{e}_u\cdot\mathbf{n}
   - \delta_v^* \tilde{e}_v\mathbf{w}\cdot\mathbf{n}
   + c^2\tilde{e}_v\nabla\delta_u^*\cdot\mathbf{n} 
  - c^2\nabla\tilde{e}_u\cdot\nabla\delta_u^*\mathbf{w}\cdot\mathbf{n} \Big].
\end{multline*}
Here we have introduced the fluxes $\delta_v^*$, $\nabla\delta_u^*$ built from $\delta_v$, $\nabla\delta_u$ according to the specification in Section \ref{fluxes}. In what follows, $C$ will be a constant independent of the solution and the element diameter $h$ for a shape-regular mesh. Here $||\cdot||$ denotes a Sobolev norm and $|\cdot|$ denotes the associated seminorm. We then have the following error estimate.
\begin{theorem}
Let $\bar{q} = \min(q-1,s)$. Then there exist numbers $C_0$, $C_1$ depending only on $s$, $q$ and the shape-regularity of the mesh, such that for smooth solutions $u$, $v$ and time $T$
\begin{multline*}
||\nabla e_{u}(\cdot,T)||_{{L^2}(\Omega)}^2 + ||e_v(\cdot,T)||_{L^2(\Omega)}^2 \\
\leq (C_0T + C_1T^2)\max_{t\leq T}\Big[h^{2\theta}\big(|u(\cdot,t)|_{H^{\bar{q}+2}(\Omega)}^2 + |v(\cdot,t)|_{H^{\bar{q}+1}(\Omega)}^2\big)\Big],
\end{multline*}  
where
\begin{eqnarray*}
 \theta = \left\{
\begin{array}{ll}
\bar{q},   \ \ \ \ \ \ \ \  \ \ \  \beta,\eta\geq0, \ \ |{\bf w}\cdot{\bf n}| \leq \frac{2c^2\xi}{c^2+\xi^2},\\
\bar{q}+\frac{1}{2},\ \ \ \ \  \ \beta,\eta > 0, \ \ |{\bf w}\cdot{\bf n}|\leq \frac{2c^2\xi}{c^2+\xi^2}.
\end{array}
\right.
\end{eqnarray*}
\end{theorem}

\begin{proof}
From the Bramble-Hilbert lemma (e.g. \ci{Ciarlet}), we have for $\bar{q}=\min (q-1,s)$ 
\begin{align*} 
 \| \delta_v \|_{L^2(\Omega)}^2 + \|\nabla\delta_u\|_{L^2(\Omega)}^2
 &\leq Ch^{2\bar{q}+2}\left(|u(\cdot,t)|_{H^{\bar{q}+2}}^2 + |v(\cdot,t)|_{H^{\bar{q}+1}(\Omega)}^2\right),\\
 \|\frac{\partial\delta_v}{\partial t}\|_{L^2(\Omega)}^2 
& \leq  Ch^{2s+2} \left| \frac{\partial v(\cdot,t)}{\partial t} \right|_{H^{s+1}(\Omega)}^2,\\
 \|\delta_v^*\|_{L^2(\partial\Omega)}^2 + \|\nabla\delta_{u}^*\cdot\mathbf{n} \|_{L^2(\partial\Omega)}^2 
 & \leq  Ch^{2\bar{q}+1}\left(|u(\cdot,t)|_{H^{\bar{q}+2}(\Omega)}^2 + |v(\cdot,t)|_{H^{\bar{q}+1}(\Omega)}^2\right), \\
 \|\tilde{e}_v\|_{L^2(\partial\Omega)}^2 + \|\nabla\tilde{e}_u\cdot\mathbf{n}\|_{L^2(\partial\Omega)}^2 
 & \leq  Ch^{-1}\left(\|\tilde{e}_v\|_{L^2(\Omega)}^2 + \|\nabla\tilde{e}_u\|_{L^2(\Omega)}^2\right). 
\end{align*}
First, consider the case where $\beta = 0$ or $\eta = 0$. From the condition (\ref{bound_w_1}) we conclude that the time derivative of the error energy satisfies 
\bd
\frac{d\mathcal{E}^h}{dt} \leq \mathcal{B}(\tilde{\mathbf{D}}_0^h, \mathbf{\Delta}^h).
\ed
 Now, by the Cauchy-Schwarz inequality we have 
\begin{multline*} 
\mathcal{B}(\tilde{\mathbf{D}}_0^h, \mathbf{\Delta}^h)  
\leq  C\sum_{j}\|\nabla\tilde{e}_u\cdot\mathbf{n}\|_{L^2(\partial\Omega_j)}\|\delta^*_v\|_{L^2(\partial\Omega_j)} 
+\|\tilde{e}_v\|_{L^2(\partial\Omega_j)}\|\nabla\delta^*_u\cdot \mathbf{n}\|_{L^2(\partial\Omega_j)}\\
+\|\nabla\tilde{e}_u\|_{L^2(\partial\Omega_j)}\|\nabla\delta^*_u\|_{L^2(\partial\Omega_j)} 
+\|\delta^*_v\|_{L^2(\partial\Omega_j)}\|\tilde{e}_v\|_{L^2(\partial\Omega_j)} \\
\leq C\sqrt{\mathcal{E}^h}h^{\bar{q}}(|u(\cdot,t)|_{H^{\bar{q}+2}(\Omega)}
+|v(\cdot,t)|_{H^{\bar{q}+1}(\Omega)}).  
\end{multline*}
Then a direct integration in time combined with the assumption that $(\tilde{e}_u, \tilde{e}_v) = 0$ at $t = 0$ gives us
\bd 
\mathcal{E}^h(T) \leq CT^2 \max_{t\leq T} h^{2\bar{q}}\left(|u(\cdot,t)|_{H^{\bar{q}+2}(\Omega)}^2+|v(\cdot,t)|_{H^{\bar{q}+1}(\Omega)}^2\right).
\ed 
For dissipative fluxes, $\beta$, $\eta > 0$, we can improve the estimates. For the boundaries of the inter-elements the contribution is
\begin{multline*}
\Theta_1 = -\sum_j \int_{F_j} \Big[ c^2[[\nabla\tilde{e}_u]]\delta_v^* 
 -[[\tilde{e}_v]]\delta_v^*\cdot\mathbf{w}
 + c^2[[\tilde{e}_v]]\cdot\nabla\delta_u^*
 - c^2\nabla\tilde{e}_{u1}\cdot\nabla\delta_u^*\mathbf{w}\cdot\mathbf{n}_1   \\
   - c^2\nabla\tilde{e}_{u2}\cdot\nabla\delta_u^*\mathbf{w}\cdot\mathbf{n}_2 \Big] 
  - \sum_j\int_{F_j} \Big[ c^{2}\eta[[\nabla\tilde{e}_u]]^2 
  + c^{2}\beta\big|[[\tilde{e}_v]]\big|^2  \\
   -c^{2}\beta[[\tilde{e}_v]]\cdot\big(\nabla\tilde{e}_{u1} 
(\mathbf{w}\cdot\mathbf{n}_1)
+\nabla \tilde{e}_{u2}(\mathbf{w}\cdot\mathbf{n}_2)\big)
-\eta [[\nabla\tilde{e}_u]][[\tilde{e}_v]]\cdot\mathbf{w} \Big], 
\end{multline*}
resulting in
\begin{multline} \label{inequality_12}
\Theta_1 \leq  C\sum_j \Big(\|\delta_v^*\|_{L^2(F_j)}^2 + \|\nabla\delta_u^*\|_{L^2(F_j)}^2\Big)  \\
 \leq  Ch^{2\bar{q}+1}\Big(|u(\cdot,t)|_{H^{\bar{q}+2}(\Omega)}^2 + |v(\cdot,t)|_{H^{\bar{q}+2}(\Omega)}^2\Big),
\end{multline}
on the physical boundaries, by using the fact 
\[
\nabla \tilde{e}_u\cdot\nabla\delta_u^* = (\nabla\tilde{e}_u\cdot{\bf n})(\nabla\delta_u^*\cdot{\bf n}) + (\nabla\tilde{e}_u)_\tau\cdot(\nabla\delta_u^*)_\tau,
\]
at inflow we have
\begin{multline*} 
\Theta_2 
 = - \sum_j \int_{B_{j_I}} \Big[ c^2\delta_v^*\nabla\tilde{e}_u\cdot\mathbf{n} 
 - \delta_v^* \tilde{e}_v\mathbf{w}\cdot\mathbf{n}
 + c^2\tilde{e}_v\nabla\delta_u^*\cdot\mathbf{n} 
 - c^2(\nabla\tilde{e}_u\cdot\mathbf{n})(\nabla\delta_u^*\cdot\mathbf{n})\mathbf{w}\cdot\mathbf{n}   \\
  -c^2(\nabla\tilde{e}_u)_{\tau} \cdot (\nabla\delta_u^*)_{\tau} \mathbf{w}\cdot\mathbf{n} \Big] 
  +\sum_j\int_{B_{j_I}} \Big[ \frac{c^2\mathbf{w}\cdot\mathbf{n}}{2} \arrowvert (\nabla\tilde{e}_u)_{\tau}\arrowvert^2 
+ \frac{c^2\mathbf{w}\cdot\mathbf{n}}{2}(\nabla\tilde{e}_u\cdot\mathbf{n})^2  \\
  + \Big(\frac{\mathbf{w}\cdot\mathbf{n}}{2}
  +\frac{(\mathbf{w}\cdot\mathbf{n})^2-c^2}{\xi-\mathbf{w}\cdot\mathbf{n}}\Big) (\tilde{e}_v)^2
  + \frac{(c^2-\xi\mathbf{w}\cdot\mathbf{n})\mathbf{w}\cdot\mathbf{n}}{\xi-\mathbf{w}\cdot\mathbf{n}}(\nabla \tilde{e}_u\cdot\mathbf{n})\tilde{e}_v \Big]. 
  \end{multline*}
Thus
\bd
\Theta _2 
 \leq  C\sum_j \Big(\|\delta_v^*\|_{L^2(B_{j_I})}^2 
+ \|\nabla\delta_u^*\|_{L^2(B_{j_I})}^2\Big)
 \leq  Ch^{2\bar{q}+1}\Big(|u(\cdot,t)|_{H^{\bar{q}+2}(\Omega)}^2 
+ |v(\cdot,t)|_{H^{\bar{q}+2}(\Omega)}^2\Big). 
\ed

Similarly, on the outflow boundaries we obtain that
\begin{multline*} 
\Theta_3 
  =  - \sum_j \int_{B_{j_O}} \Big[ c^2\delta_v^*\nabla\tilde{e}_u\cdot\mathbf{n} 
  - \delta_v^* \tilde{e}_v\mathbf{w}\cdot\mathbf{n}
  + c^2\tilde{e}_v\nabla\delta_u^*\cdot\mathbf{n} 
  - c^2(\nabla\tilde{e}_u\cdot\mathbf{n})(\nabla\delta_u^*\cdot\mathbf{n})\mathbf{w}\cdot\mathbf{n}   \\
  - c^2(\nabla\tilde{e}_u)_{\tau} \cdot (\nabla\delta_u^*)_{\tau} \mathbf{w}\cdot\mathbf{n} \Big] 
 + \sum_j\int_{B_{j_O}} \Big[ -\frac{c^2}{2} \arrowvert (\nabla \tilde{e}_u)_{\tau} \arrowvert^2\mathbf{w}\cdot\mathbf{n}   \\
  +(\frac{c^2}{2\xi}+\frac{\xi}{2}) (\nabla \tilde{e}_u\cdot\mathbf{n})\tilde{e}_v\mathbf{w}\cdot\mathbf{n}
 - \frac{c^2\xi}{2}(\nabla \tilde{e}_u\cdot\mathbf{n})^2
 - \frac{c^2}{2\xi}(\tilde{e}_v)^2 \Big].
 \end{multline*}
Then
\begin{multline} \label{inequality_14}
\Theta_3  \leq  C\sum_j \Big[ \|\delta_v^*\|_{L^2(B_{j_O})}^2 
 + \|\nabla\delta_u^*\|_{L^2(B_{j_O})}^2\Big] \\
 \leq  Ch^{2\bar{q}+1}\Big[|u(\cdot,t)|_{H^{\bar{q}+2}(\Omega)}^2 
 + |v(\cdot,t)|_{H^{\bar{q}+2}(\Omega)}^2\Big].
\end{multline}
Combining (\ref{derivative_error_energy}) with (\ref{inequality_12})-(\ref{inequality_14}) yields 
\bd
\mathcal{E}^h(T) \leq CT \max_{t\leq T} h^{2\bar{q}+1}(|u(\cdot,t)|_{H^{\bar{q}+2}(\Omega)}^2+|v(\cdot,t)|_{H^{\bar{q}+1}(\Omega)}^2).
\ed
\end{proof}

\begin{rem}
A similar analysis yields the same results in the presence of supersonic boundaries, $\arrowvert \mathbf{w} \cdot \mathbf{n} \arrowvert > c$.
\end{rem}

\subsection{Improved estimates for $m = 1$}
We can improve this estimate if we only consider the 1d case. Now assume $s = q-1$ and seek $(\tilde{u}^h, \tilde{v}^h)$ such that the boundary terms in $\mathcal{B}(\tilde{\mathbf{D}}_0^h, \mathbf{\Delta}^h)$ vanish:
\be \label{fluxes_1d}
\delta_v^* = \frac{\partial\delta_u^*}{\partial x} = 0.
\ee
This can be accomplished if we enforce the boundary condition on the end points of the element $\Omega_j = (x_{j-1}, x_j)$
\begin{eqnarray} \label{left_bc}
(1+\beta-\alpha)\delta_{v} + (\eta+\alpha)\frac{\partial \delta_{u}}{\partial x}= 0, &\ \ &    x = x_{j-1}, \\
(\beta+\alpha)\delta_{v}  - (1+\eta-\alpha)\frac{\partial\delta_{u}}{\partial x} = 0,  & \ \ &  x = x_{j}.  \label{right_bc}
\end{eqnarray}
As shown in \ci{DGwave}, we find we must assume
\bd 
\alpha(1-\alpha) = \beta\eta.
\ed 
This will be satisfied by the Sommerfeld flux but it does not hold for the central flux.  
Given (\ref{left_bc}) we construct $\delta_u$ and $\delta_v$ by requiring
\be  \label{restriction_polynomials_1d}
\int_{x_{j-1}}^{x_j}\phi\frac{\partial\delta_u}{\partial x} = \int_{x_{j-1}}^{x_j}\phi\delta_v = \int_{x_{j-1}}^{x_j}\delta_u = 0,
\ee
where $\phi$ is an arbitrary polynomial of degree $q-2$. Using the Bramble-Hilbert lemma, for $(u,v)\in H^{q+2}(\Omega)\times H^{q+1}(\Omega)$ we have the following inequality
\begin{eqnarray} \label{lemma_bound}
\Big|\Big|\frac{\partial\delta_u}{\partial t}\Big|\Big|_{H^1(\Omega)} + \Big|\Big|\frac{\partial\delta_v}{\partial t}\Big|\Big|_{L^2(\Omega)}\leq Ch^{q}\Big(\Big|\frac{\partial u}{\partial t}\Big|_{H^{q+1}(\Omega)}+\Big|\frac{\partial v}{\partial t}\Big|_{H^{q}(\Omega)}\Big).
\end{eqnarray}
Now, repeating the computations from the previous section and invoking
(\ref{fluxes_1d}) and (\ref{restriction_polynomials_1d}) yields
\bd 
\mathcal{B}(\tilde{\mathbf{D}}_0^h, \mathbf{\Delta}^h) = \sum_j\int_{x_{j-1}}^{x_j} c^2\frac{\partial\tilde{e}_u}{\partial x}\frac{\partial}{\partial x}\left(\frac{\partial \delta_u}{\partial t}\right)  + \tilde{e}_v\frac{\partial \delta_v}{\partial t}.
\ed 
Then (\ref{lemma_bound}) gives us the improved estimate
\bd
\frac{d\mathcal{E}^h}{dt}\leq Ch^{q}\sqrt{\mathcal{E}^h}\Big(|u(\cdot,t)|_{H^{q+2}(\Omega)}^2 + |v(\cdot,t)|_{H^{q+1}(\Omega)}^2\Big)^{1/2}.
\ed

\section{Numerical experiments}\label{experiments}
In this section, we present some numerical results
to study the convergence in the $L_2$ norm for our method. In the experiments
we add a forcing term, f, to the equations. Such a term could be incorporated
into the previous analysis without changing the results. In all cases we used a standard
modal formulation with a tensor-product Legendre basis and marched in time
using the 4-stage fourth order Runge-Kutta scheme (RK4).

For the experiments we choose a time step sufficiently small to make the errors
due to the spatial discretization dominate. We note that a study of the
spectrum of the spatial discretization establishes that its spectral radius
scales with $(c+\arrowvert \mathbf{w} \arrowvert) \frac {q^2}{h}$, with some variability depending on whether $q$ is even or odd. This is comparable to
what was found in the case of the scalar wave equation \ci{DGwave}.


\subsection{Periodic boundary conditions in one space dimension}
To investigate the order of accuracy of our methods, we solve
\bd
u_{tt} + 2wu_{tx} + w^{2}u_{xx} = c^{2}u_{xx}, \ \ x\in (0, 1), \ \ t \geq 0,
\ed
with the initial condition
\bd %
u(x, 0) = \sin(2\pi x), \ \ x\in(0, 1),
\ed
and periodic boundary condition $u(0, t) = u(1, t)$ for $t\geq 0$. This problem has the exact solution which is a traveling wave
\bd 
u(x, t) = \cos(2c\pi t)\sin(2\pi(x-wt)),\ \ t\geq 0.
\ed

The discretization is performed on a uniform mesh with element vertices $x_{i} = ih$, $i = 0,\ldots, n$, $h = 1/n$. We evolve the solution until $T = 0.4$ with time step  $\Delta t = {\rm CFL}\times h$ for the degree of approximation polynomials $q = (1,2,3,4,5,6)$. We present the $L^{2}$-error for both $u^h$ and $v^h$.  

In the numerical experiments, we test two different fluxes: the central flux and the upwind flux. We present three different cases: $|w| = c$,  $|w| < c$, and $|w| > c$. These choices are consistent with our theory. Note that if $|w| > c$ the upwind flux is taken from a single element. 

We also consider two different choices for the degrees of the approximation spaces: either the approximation degree of $v^h$ is one less than the approximation degree of $u^h$ or $u^h$ and $v^h$ are in the same space. To remove the effect of the temporal error we set,  for the central flux, ${\rm CFL} = 0.075/(2\pi)$ when $q = (1,2,3,4,5)$ and ${\rm CFL} = 0.00375/(2\pi)$ when $q = 6$. For the upwind flux with $|w| < c$, we set ${\rm CFL} = 0.1125/(2\pi)$ when $q = (1,2,3,4,5)$, and ${\rm CFL} = 0.01125/(2\pi)$ when $q = 6$. Finally, for the upwind flux with $|w|>c$, we set ${\rm CFL} = 0.075/(2\pi)$ when $q = (1,2,3,4,5)$, and ${\rm CFL} = 0.0075/(2\pi)$ when $q = 6$.

In our initial numerical experiments we found that the convergence was somewhat irregular in
all cases when we used $L^2$-projection to determine the initial conditions. Possibly this could be remedied for the upwind flux by using the special projection required by the analysis, see for example the approach in \ci{ChouShuXing2014} which discusses a projection for the LDG method with alternating fluxes. Here we propose a simpler solution which is to transform the problem
to one with zero initial data:
\bd 
    u(x,t) = \tilde{u}(x,t)+u_{0}(x)e^{-t^2},
\ed
where $u_{0}(x)$ is the initial condition for $u(x)$, and then numerically solve for $\tilde{u}$.

\begin{figure}[tbhp]
\begin{center}
\subfloat[$w = 0.5, c = 1$]{\label{fig:a1}\includegraphics[width=0.5\textwidth]{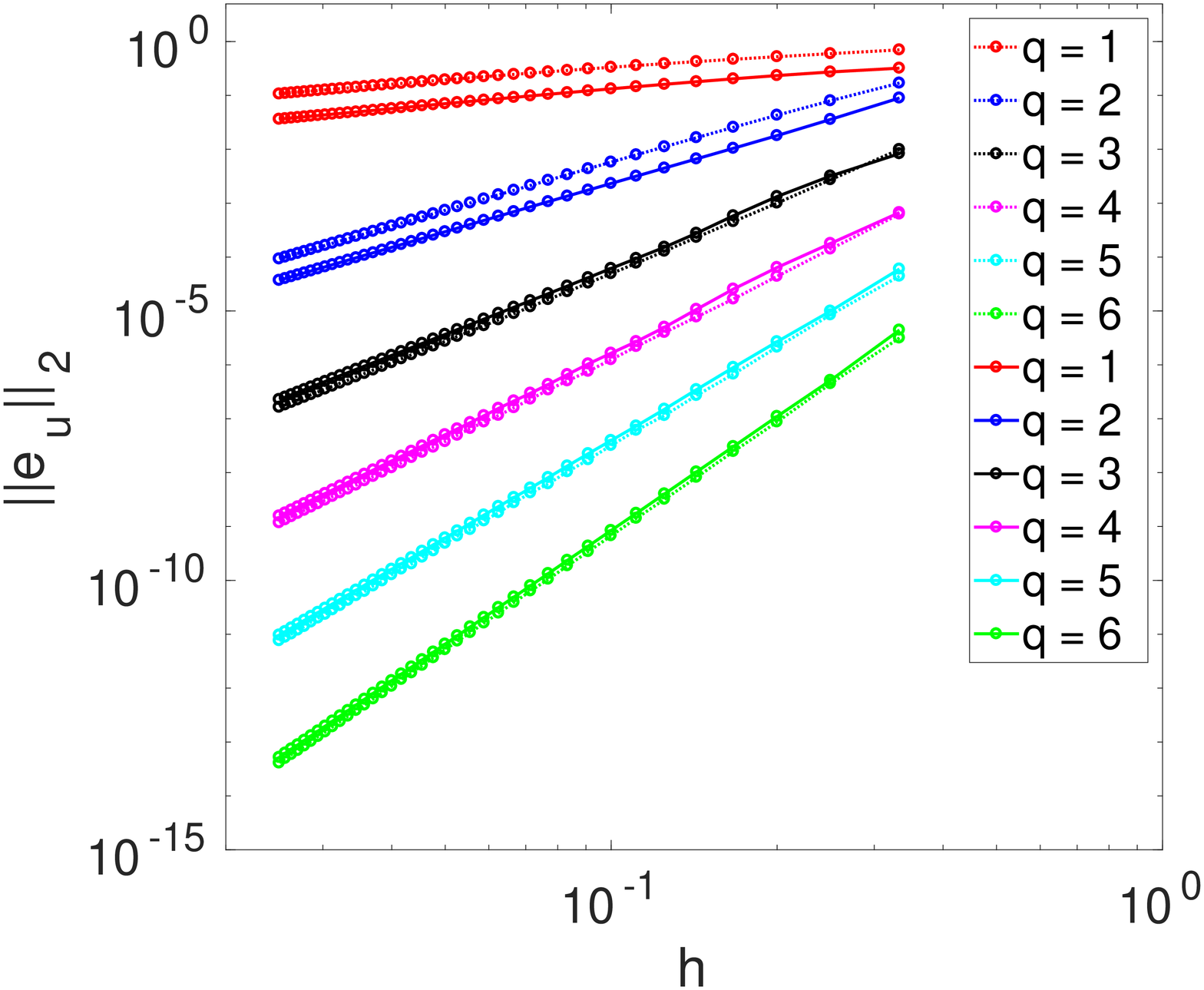}}
\subfloat[$w = 0.5, c = 1$]{\label{fig:d1}\includegraphics[width=0.5\textwidth]{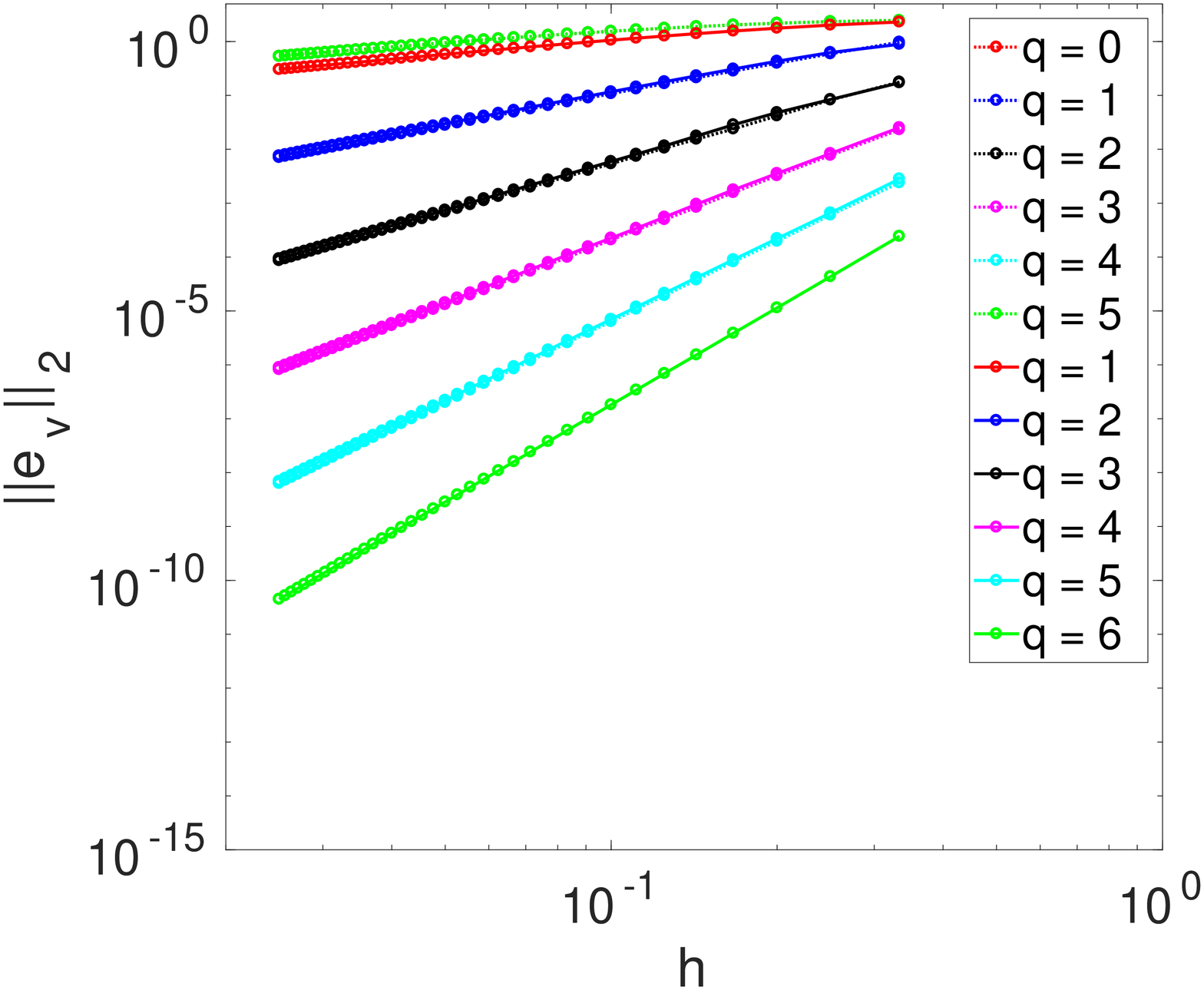}}\\
\subfloat[$w = 1, c = 0.5$]{\label{fig:c1}\includegraphics[width=0.5\textwidth]{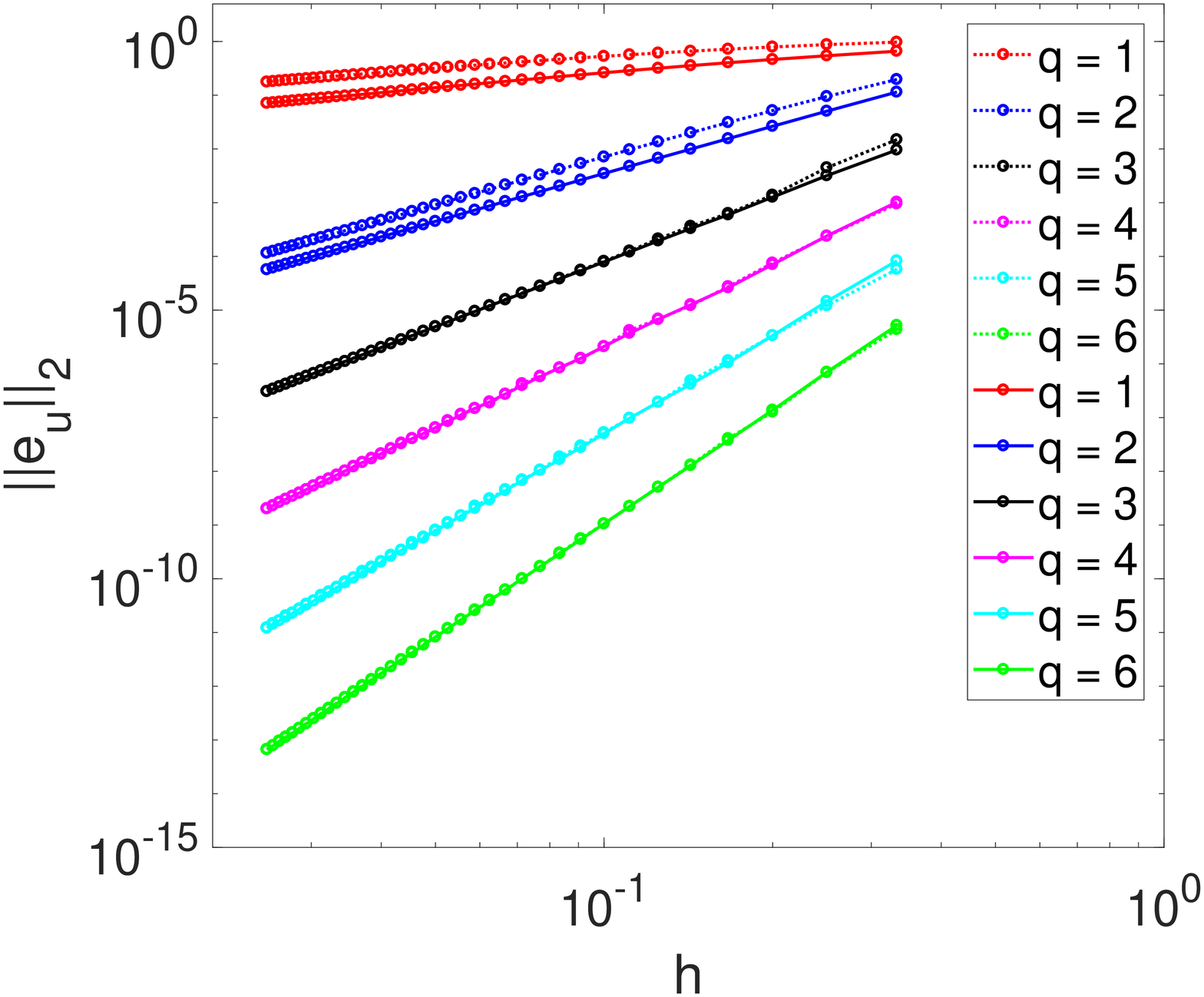}}
\subfloat[$w = 1, c = 0.5$]{\label{fig:f1}\includegraphics[width=0.5\textwidth]{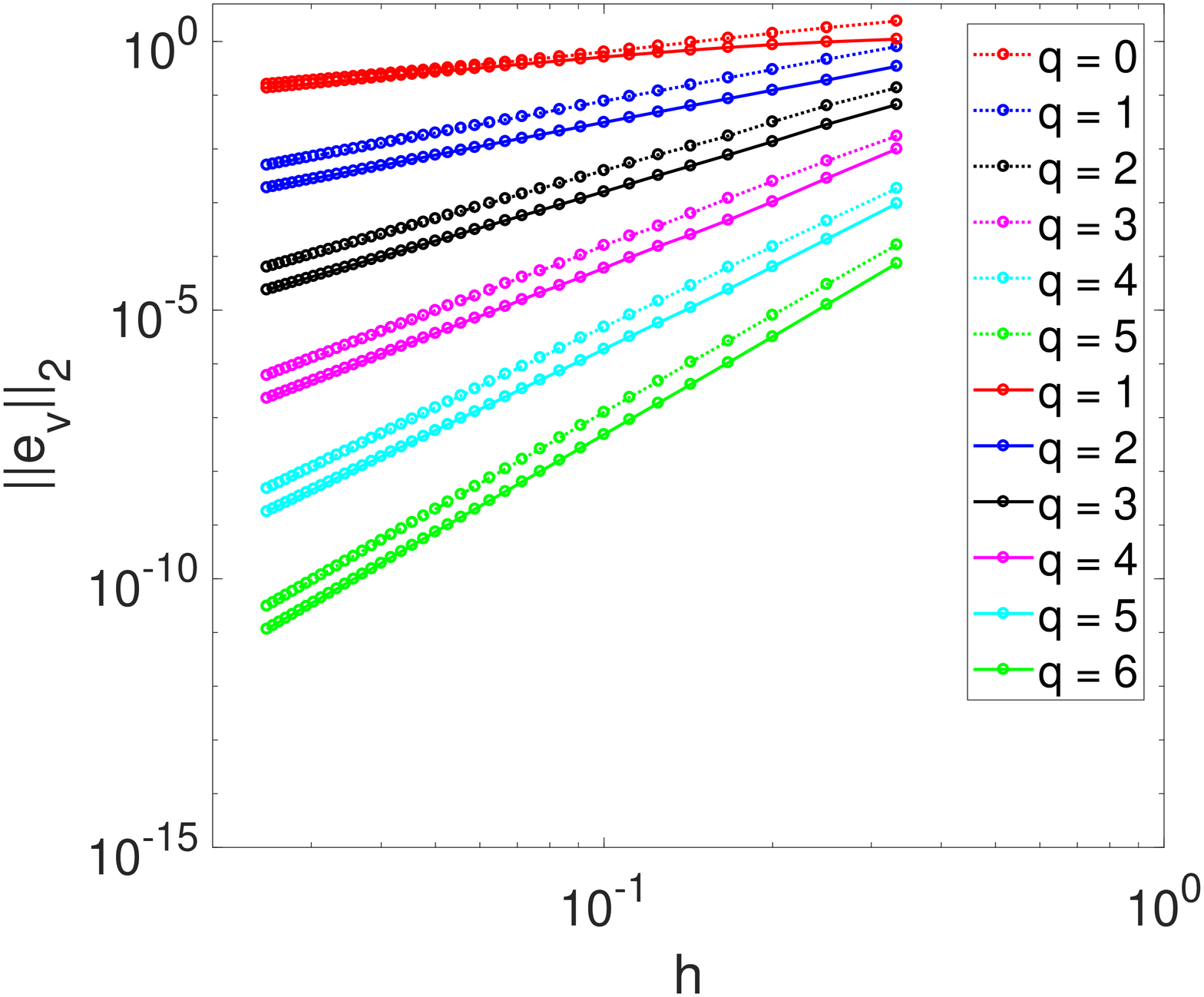}}\\
\subfloat[$w = c = 0.5$]{\label{fig:b1}\includegraphics[width=0.5\textwidth]{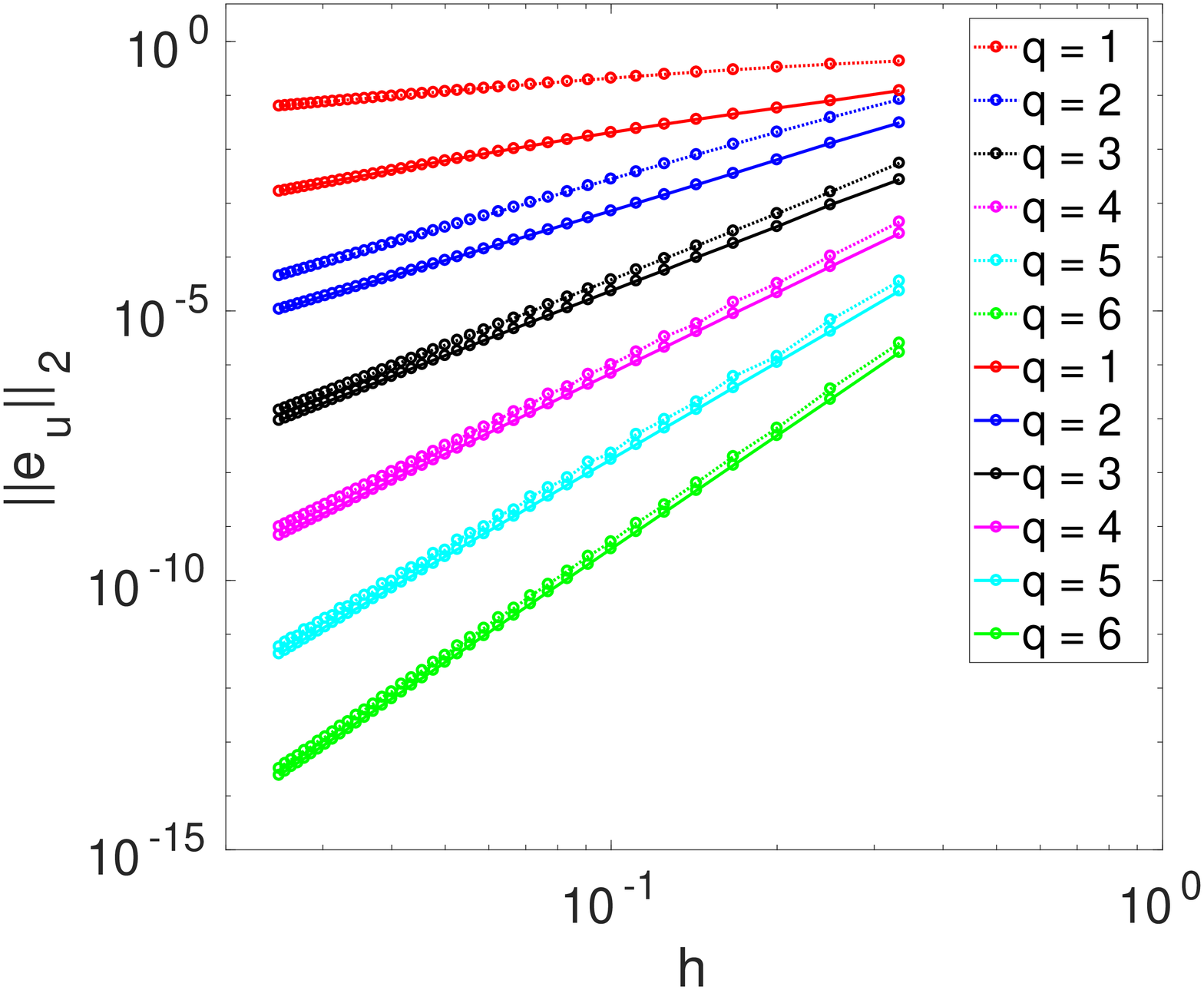}}
\subfloat[$w = c = 0.5$]{\label{fig:e1}\includegraphics[width=0.5\textwidth]{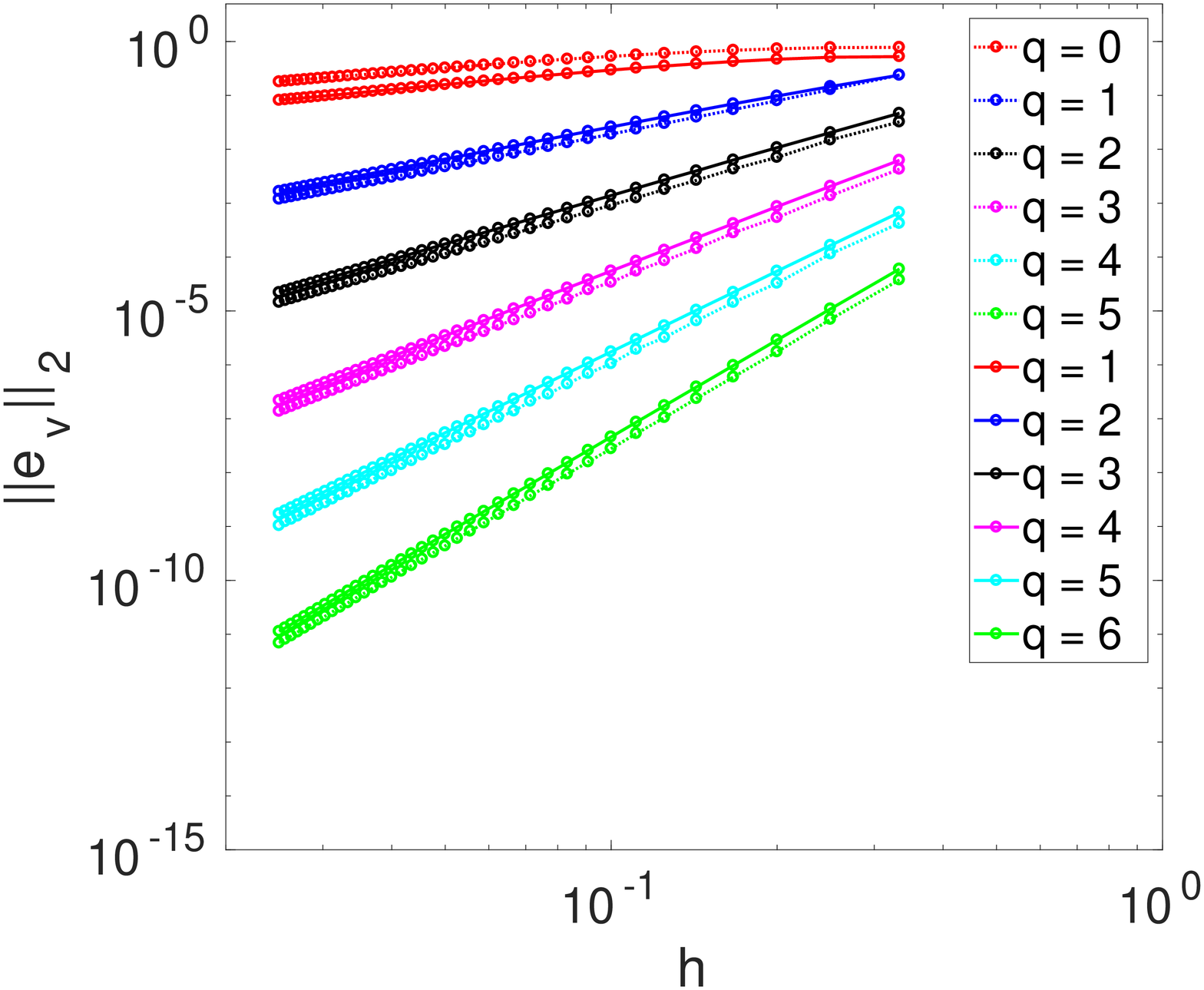}}
\caption{Plots of the error in $u$ (left column) and $v$
(right column) as a function of $h$ in 1d with upwind flux for periodic boundary condition. In the legend, $q$ is the maximum degree of the approximation of $u$, solid lines represent the case of $u^h$ and $v^h$ in the same space, dotted lines represent the case of $v^h$ one degree lower than $u^h$.}
\label{fig:uv_up_1d}
\end{center}
\end{figure}

The $L^2$ error for $u$ and $v$ are plotted against the grid spacing $h$ in Figure \ref{fig:uv_up_1d} for both $u^h$ and $v^h$ when the upwind flux is used.
Linear regression estimates of the rate of convergence, for $u^h$ and $v^h$ in the same polynomial space, can be found in Table \ref{R_up_per_1d_s}, and  for the degree of $v^h$
one less than that of $u^h$, in Table \ref{R_up_per_1d_d}. Note that we only use the ten finest grids to obtain the rates of convergence. 

For $q \geq 2$ we observe the same rate of convergence, $q+1$ for $u$ and $q$ for $v$, for the two choices of approximation space for $v$. However, from the graphs we see that there are sometimes noticeable differences in accuracy. Generally speaking, errors are smaller when $v^h$ is taken from the same space as $u^h$, the only exception being the
errors in approximating $v$ for the rather special case of $w=c$. 
\begin{table}[htb]
\caption{Linear regression estimates of the convergence rate of $u$ and $v$ in $1d$ with upwind flux for periodic boundary condition, approximation for $v$ is one degree lower than $u$.\label{R_up_per_1d_d}}
\begin{center}
  \begin{tabular}{|l|c c c c c c|}
  \hline
Degree $(q)$ of approx. for $u$   & 1 & 2 & 3 & 4 & 5 & 6 \\
\hline
  Rate fit $u$ $  (w = 0.5, c = 1)$ &0.90 & 3.00 & 4.05& 5.03 & 5.92 & 6.91 \\
  Rate fit $v$ $  (w = 0.5, c = 1)$ &0.87 & 1.99 & 2.99 & 3.99 & 5.00 & 6.00 \\
\hline
 Rate fit $u$ $ (w = 0.5, c = 0.5)$ &0.92& 3.00 & 4.01 & 5.00 & 6.14 & 7.00 \\
 Rate fit $v$ $ (w = 0.5, c = 0.5)$ &0.88 & 2.00 & 3.00 & 4.00 & 5.06 & 5.99  \\
  \hline 
   Rate fit $u$ $ (w = 1, c = 0.5)$ &0.88& 2.99 & 4.01 & 5.03 & 6.04 & 6.93 \\
  Rate fit $v$ $ (w = 1, c = 0.5)$ &0.93 & 1.99 & 2.99 & 3.99 & 5.00 & 6.00  \\
  \hline
\end{tabular}
\end{center}
\end{table}

\begin{table}[htb]
\caption{Linear regression estimates of the convergence rate of $u$ and $v$ in $1d$ with upwind flux for periodic boundary condition, $u$ and $v$ are in the same approximation space.
\label{R_up_per_1d_s}}
\begin{center}
  \begin{tabular}{|l|c c c c c c|}
  \hline
Degree $(q)$ of approx. of $u$   & 1 & 2 & 3 & 4 & 5 & 6 \\
\hline
  Rate fit $u$ $ (w = 0.5, c = 1)$ &0.97& 3.00 & 4.01& 5.00 & 5.98 & 6.95 \\
  Rate fit $v$ $ (w = 0.5, c = 1) $&0.95 & 1.99 & 3.00 & 4.00 & 5.00 & 6.00 \\
\hline
 Rate fit $u$ $ (w = 0.5, c = 0.5)$ &1.91& 3.01 & 4.00 & 5.00 & 6.00 & 6.89 \\
 Rate fit $v$ $ (w = 0.5, c = 0.5)$ &0.98 & 2.00 & 3.00 & 4.00 & 5.00 & 6.00  \\
 \hline
  Rate fit $u$ $ (w = 1, c = 0.5)$ & 0.97& 2.99 & 4.00 & 5.01 & 5.99 & 6.90 \\
  Rate fit $v$ $ (w = 1, c = 0.5)$ & 0.99 & 2.02 & 3.02 & 4.01 & 5.00 & 6.01  \\
\hline
\end{tabular}
\end{center}
\end{table}

\begin{figure}[tbhp]
\begin{center}
\subfloat[$w = 0.5, c = 1$]{\label{fig:a1c}\includegraphics[width=0.5\textwidth]{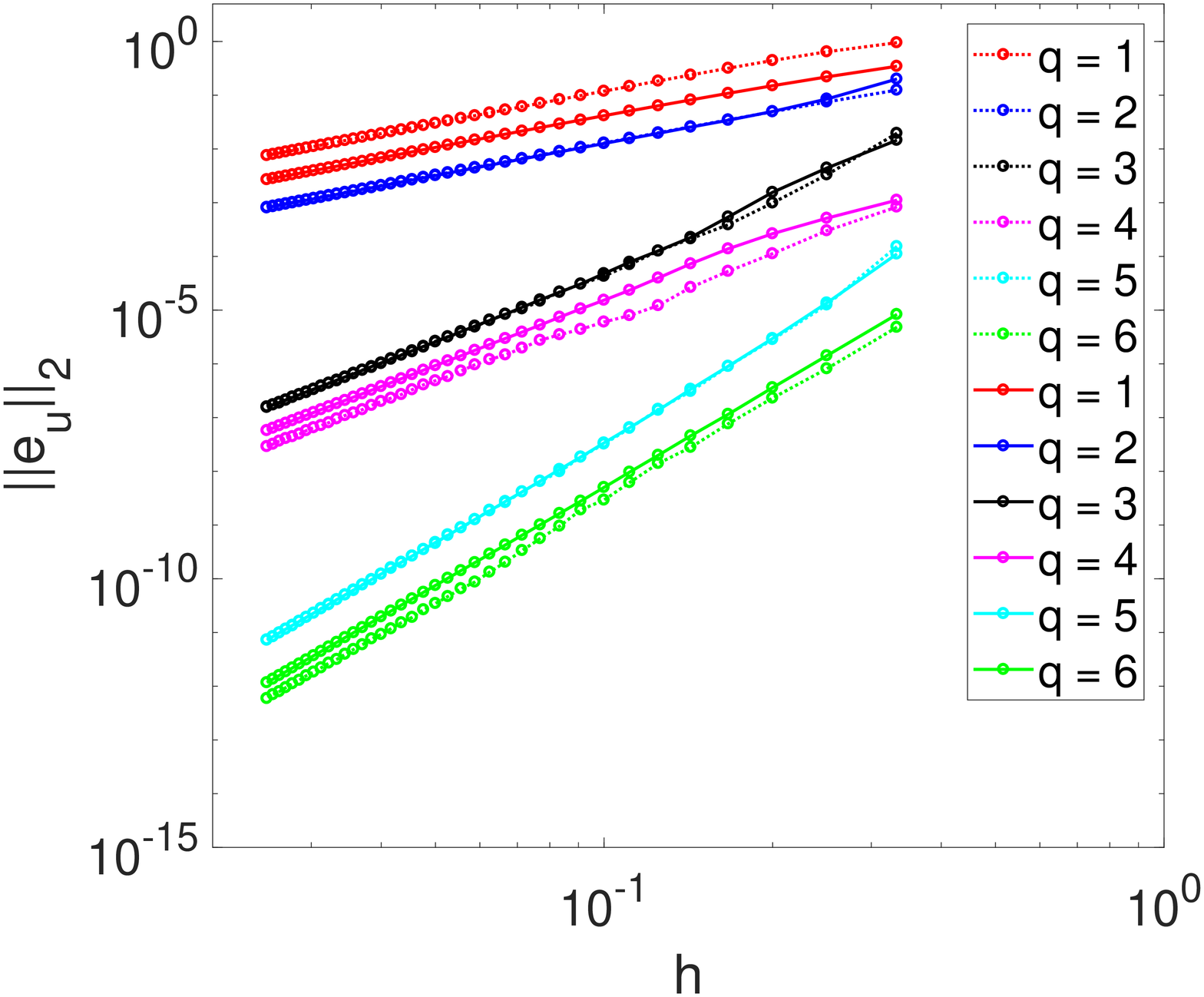}}
\subfloat[$w = 0.5, c = 1$]{\label{fig:va1c}\includegraphics[width=0.5\textwidth]{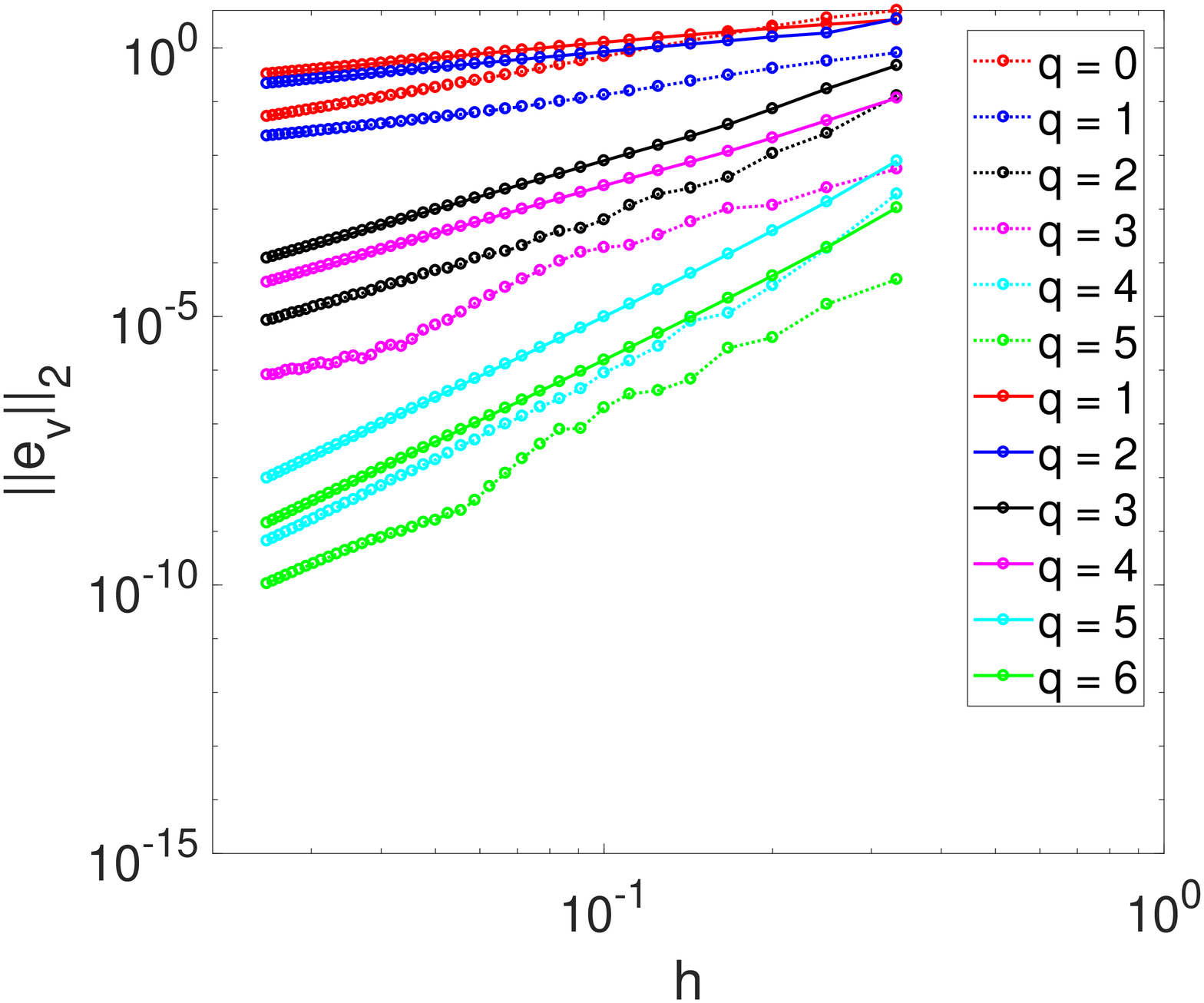}}\\
\subfloat[$w = 1, c = 0.5$]{\label{fig:c1c}\includegraphics[width=0.5\textwidth]{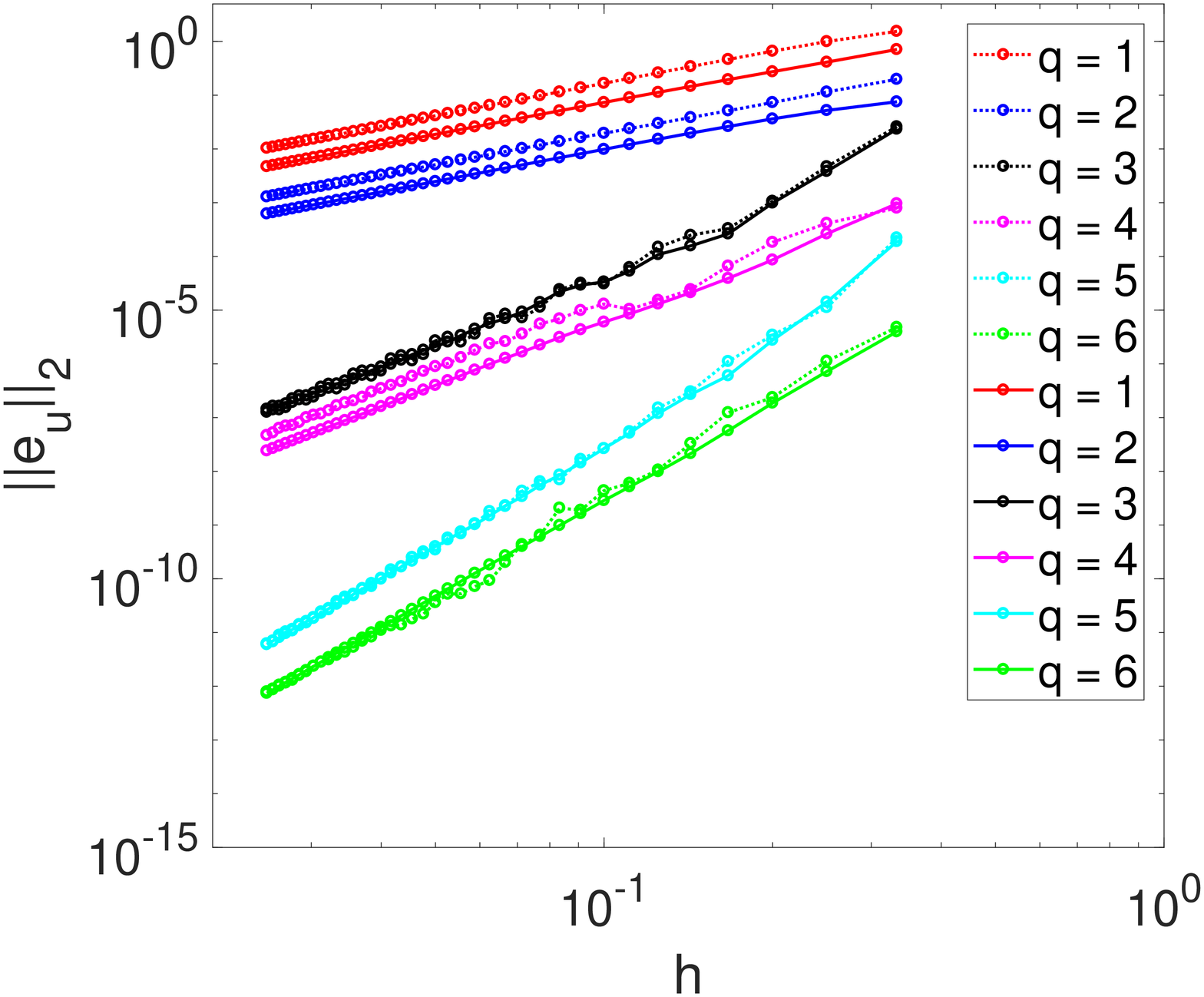}}
\subfloat[$w = 1, c = 0.5$]{\label{fig:vc1c}\includegraphics[width=0.5\textwidth]{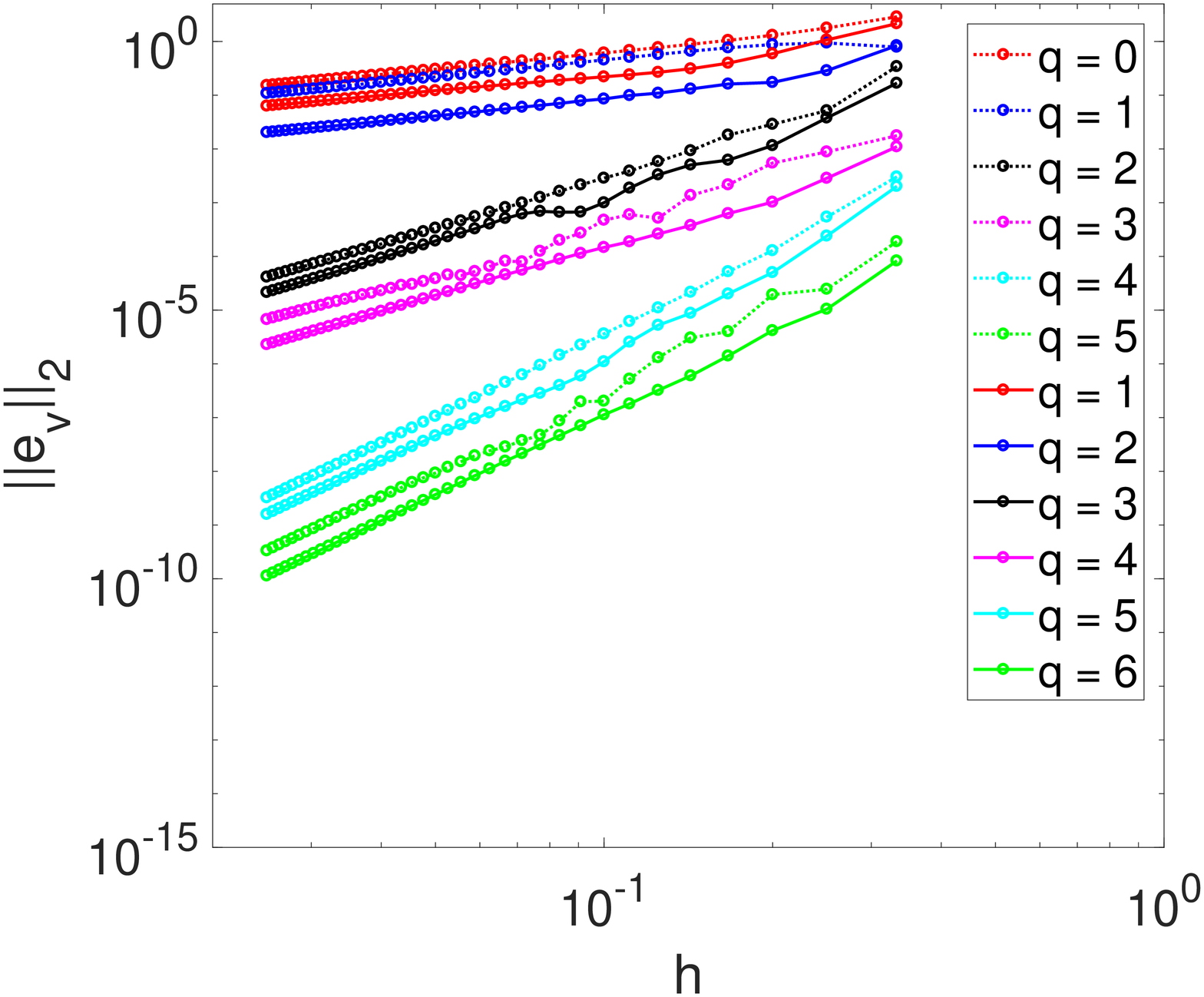}}\\
\subfloat[$w = c = 0.5$]{\label{fig:b1c}\includegraphics[width=0.5\textwidth]{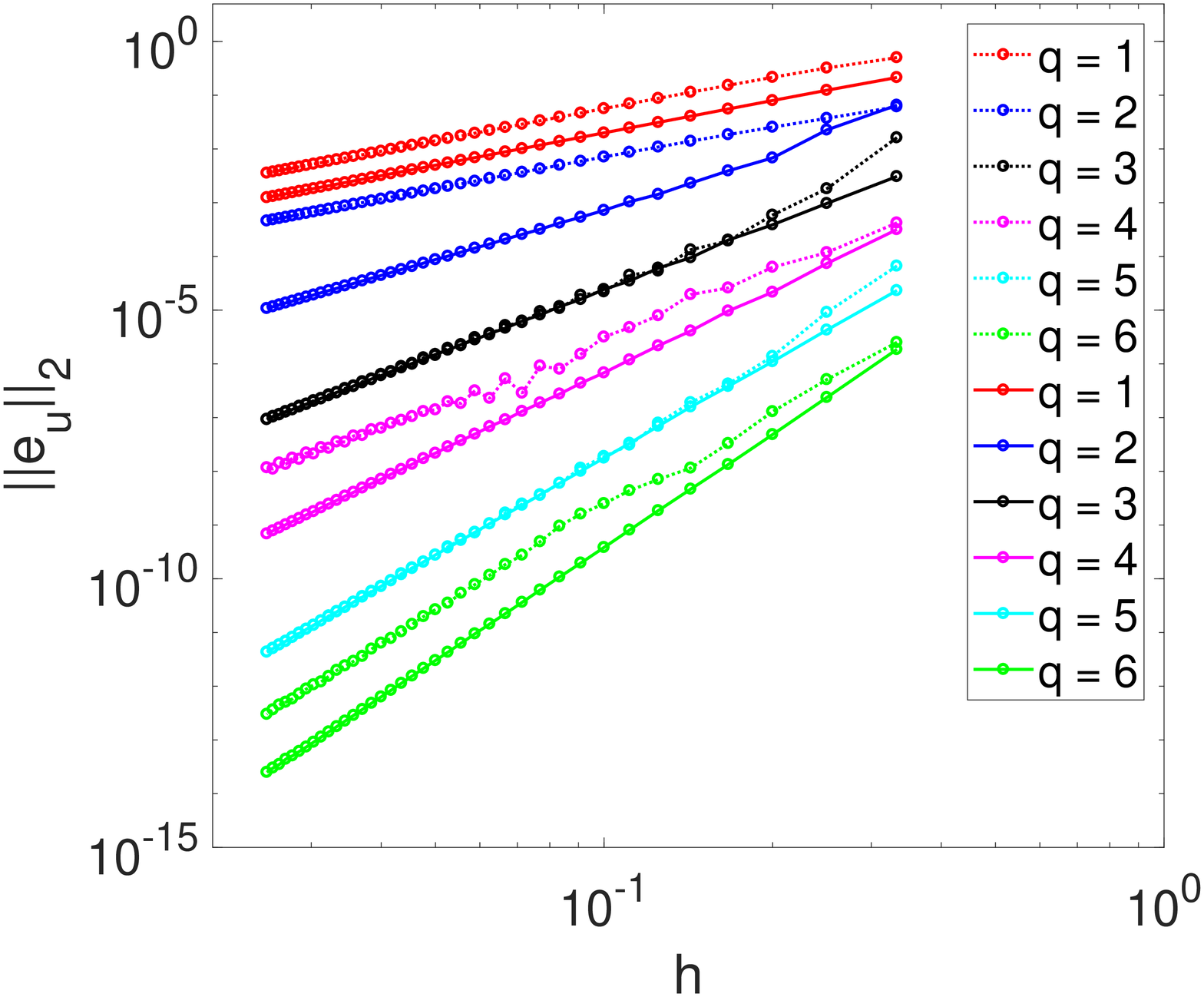}}
\subfloat[$w = c = 0.5$]{\label{fig:vb1c}\includegraphics[width=0.5\textwidth]{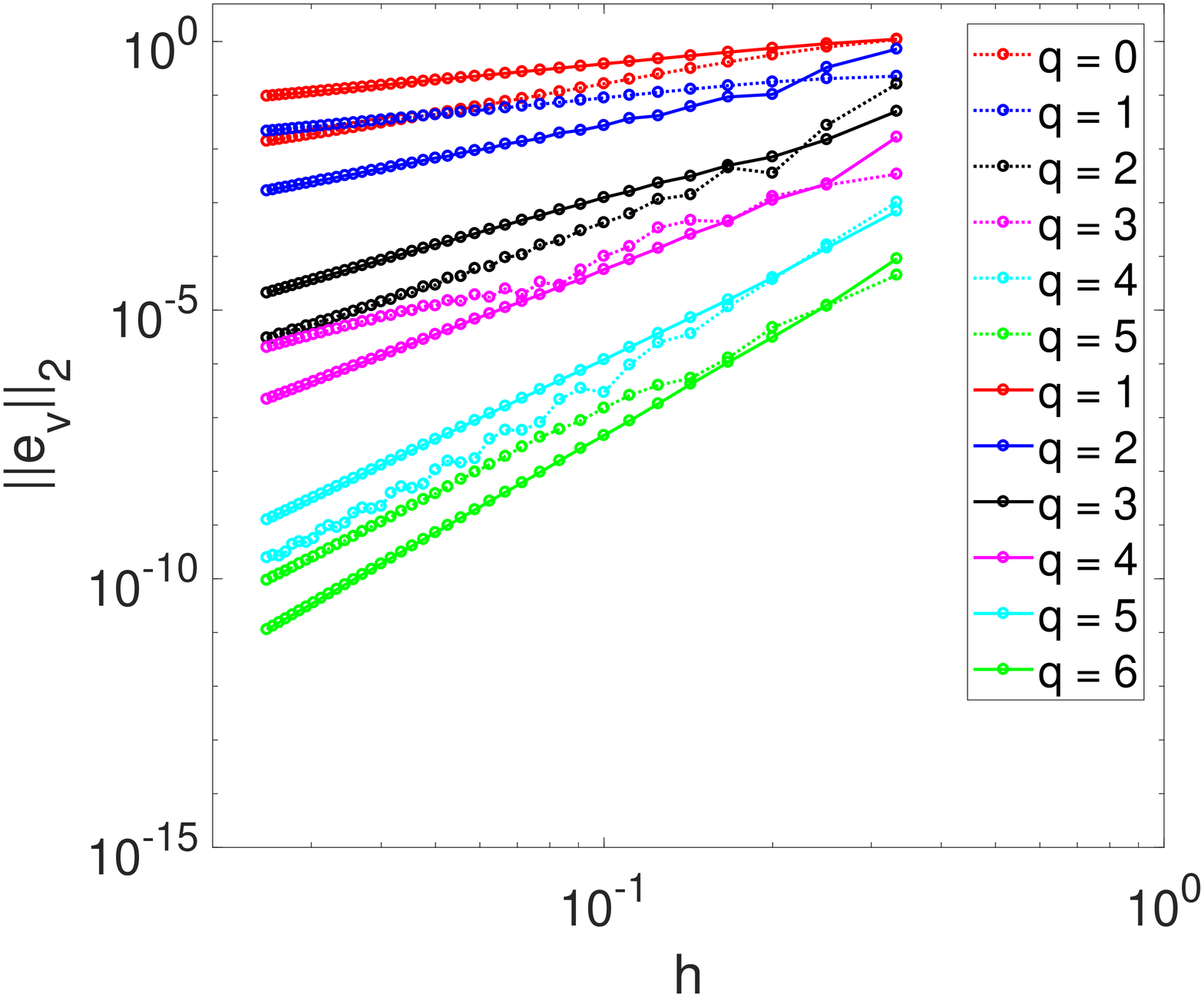}}
\caption{Plots of the error in $u$ (left column) and $v$
(right column) as a function of $h$ in 1d with central flux for periodic boundary condition. In the legend, $q$ is the maximum degree of the approximation of $u$, solid lines represent the case of $u^h$ and $v^h$ in the same space, dotted lines represent the case of $v^h$ one degree lower than $u^h$.\label{fig:uv_center_1d}}
\end{center}
\end{figure}

In Figure \ref{fig:uv_center_1d} the $L^2$ errors in $u$ and $v$ are plotted against the grid-spacing $h$ for the central flux. 
Linear regression estimates of the convergence rate can be found in Table \ref{R_center_per_1d_s} for $u^h$ and $v^h$ in the same approximation space and in Table \ref{R_center_per_1d_d} for $u^h$ and $v^h$ in different spaces.

Excluding the special case $\arrowvert w \arrowvert=c$, we observe for $q$ odd, optimal convergence, $q+1$, for $u$ while the rate of convergence for $v$ is one order lower than $u$. When $u$ and $v$ are in the same space this is suboptimal for $v$. For even $q$ the rate of convergence is only $q$ for $u$. The convergence rate for $v$ is always one
less than for $u$. 
\begin{table}[htb]
\caption{Linear regression estimates of the convergence rate for $u$ and $v$ in $1d$ with central flux for periodic boundary condition, the approximation for $v$ is one degree lower than $u$.\label{R_center_per_1d_d}}
\begin{center}
  \begin{tabular}{|l|c c c c c c|}
  \hline
 Degree $(q)$ of approx. of $u$   & 1 & 2 & 3 & 4 & 5 & 6 \\
\hline
  Rate fit $u$ $(w = 0.5, c = 0.5)$ & 2.00 & 1.99 & 4.05& 3.71 & 6.01 & 6.27 \\
  Rate fit $v$ $(w = 0.5, c = 0.5)$ & 1.61 & 1.01 & 3.24 & 2.82 & 5.40 & 5.27 \\
\hline
 Rate fit $u$ $(w = 0.5, c = 1) $ &2.00& 2.00 & 4.03 & 4.03 & 5.99 & 5.91 \\
 Rate fit $v$ $(w = 0.5, c = 1) $ &1.72 & 1.09 & 3.03 & 2.05 & 5.06 & 4.48  \\
\hline
  Rate fit $u$ $(w = 1, c = 0.5) $ &2.00& 1.99 & 4.13 & 4.11 & 5.81 & 5.60 \\
  Rate fit $v$ $(w = 1, c = 0.5) $ &1.00 & 1.01 & 3.01 & 2.74 & 5.02 & 4.95  \\
\hline
\end{tabular}
\end{center}
\end{table}

\begin{table}[htb]
\caption{Linear regression estimates of the convergence rate of $u$ and $v$ in $1d$ with central flux for periodic boundary condition, $u$ and $v$ are in the same approximation space.\label{R_center_per_1d_s}}
\begin{center}
  \begin{tabular}{|l|c c c c c c|}
  \hline
Degree $(q)$ of approx. of $u$   & 1 & 2 & 3 & 4 & 5 & 6 \\
\hline
Rate fit $u$ $(w = 0.5, c = 0.5)$ &2.00 & 3.01 & 3.99& 4.99 & 6.00 & 6.73 \\
Rate fit $v$ $(w = 0.5, c = 0.5)$ &1.00 & 2.01 & 2.99 & 3.97 & 4.99 & 6.01 \\
 \hline
 Rate fit $u$ $(w = 0.5, c = 1 )$ & 1.99 & 2.00 & 4.03 & 4.01 & 6.02 &6.01\\
 Rate fit $v$ $(w = 0.5, c = 1 )$ & 0.99 & 1.00 & 3.01 & 3.00 & 5.00 & 5.01  \\
 \hline
  Rate fit $u$ $(w = 1, c = 0.5) $ & 2.00& 2.00 & 4.01 & 4.03 & 6.04 & 6.01 \\
  Rate fit $v$ $(w = 1, c = 0.5) $ & 0.97 & 1.01 & 3.12 &  3.03&  4.91& 5.02  \\
 \hline
\end{tabular}
\end{center}
\end{table}

\subsection{Periodic boundary conditions in two space dimensions}
We now test our method on the problem
\bd
(\frac{\partial }{\partial t} + \textbf{w}\cdot\nabla)^{2}u = c^{2}\Delta u, \ \ (x,y)\in(0,1)\times(0,1), \ \ t > 0,
\ed
with periodic boundary conditions $u(0,y,t) = u(1,y,t)$, $u(x,0,t) = u(x,1,t)$ for $t\geq 0$. We approximate the exact solution
\bd
u(x,y,t) = \sin(2c\pi t)\Big(\sin\big(2\pi(x-w_{x}t)\big) +\sin\big(2\pi(y-w_{y}t)\big)\Big), \ \ t \geq 0.
\ed

\begin{table}
\caption{Linear regression estimates of the convergence rate of $u$ and $v$ in $2d$ with upwind flux for periodic boundary condition and $q_{x} = q_{y} = q$.
\label{R_O_per_LR2}}
\begin{center}
  \begin{tabular}{|l|c c c c c c|}
  \hline
Degree $(q)$ of approx. of $u$ and $v$  & 1 & 2 & 3 & 4 & 5 & 6 \\
\hline
  Rate fit $u$ $ (w_{x}=1,w_{y}=1,c=1)$&1.77 & 3.04 & 3.99& 5.00 & 6.00 & 6.97 \\
  Rate fit $v$ $ (w_{x}=1,w_{y}=1,c=1)$&0.89 & 1.96 & 2.97& 3.98 & 4.98 & 5.99 \\
\hline 
 Rate fit $u$ $ (w_{x}=0.5,w_{y}=1.5,c=1)$&1.05& 2.93 & 4.00 & 4.99 & 5.99 & 6.96 \\
 Rate fit $v$ $ (w_{x}=0.5,w_{y}=1.5,c=1)$&0.90& 1.91 & 2.99 & 3.97 & 4.98 & 5.99 \\
\hline  
   Rate fit $u$  $(w_{x}=0.5,w_{y}=0.5,c=1)$&1.07& 2.95 & 4.02 & 4.98 & 5.99 & 7.00 \\
   Rate fit $u$  $(w_{x}=0.5,w_{y}=0.5,c=1)$&0.89& 1.92 & 2.97 & 3.97 & 4.98 & 5.98 \\
\hline 
\end{tabular}
\end{center}
\end{table}

The discretization is performed with elements whose vertices are on the Cartesian grid defined by $x_{i} = ih $, $y_{j} = jh $, \mbox{$i,j = 0,1,\ldots,n$} with $h = 1/n$. Here we restrict attention to the case where $u^{h}$ and $v^h$ are in the same space. We evolve the solution until $T = 0.2$ using the classic fourth order Runge-Kutta method and with the time step size $\Delta t = {\rm CFL} h $.

In the numerical experiments we test both the central flux and the upwind flux. We have ${\rm CFL} = 0.075/(2\pi)$ for the central flux and ${\rm CFL = 0.0375/(2\pi)}$ for the upwind flux. Note that at an interface with supersonic normal flow the upwind flux is one-sided. Also, we only display graphs of the error in $u$, but tabulate the convergence rates for both variables. 

\begin{figure}[htb]
\begin{center}
\subfloat[$w_{x} = w_{y} = c = 1$]{\label{fig:e}\includegraphics[width=0.5\textwidth]{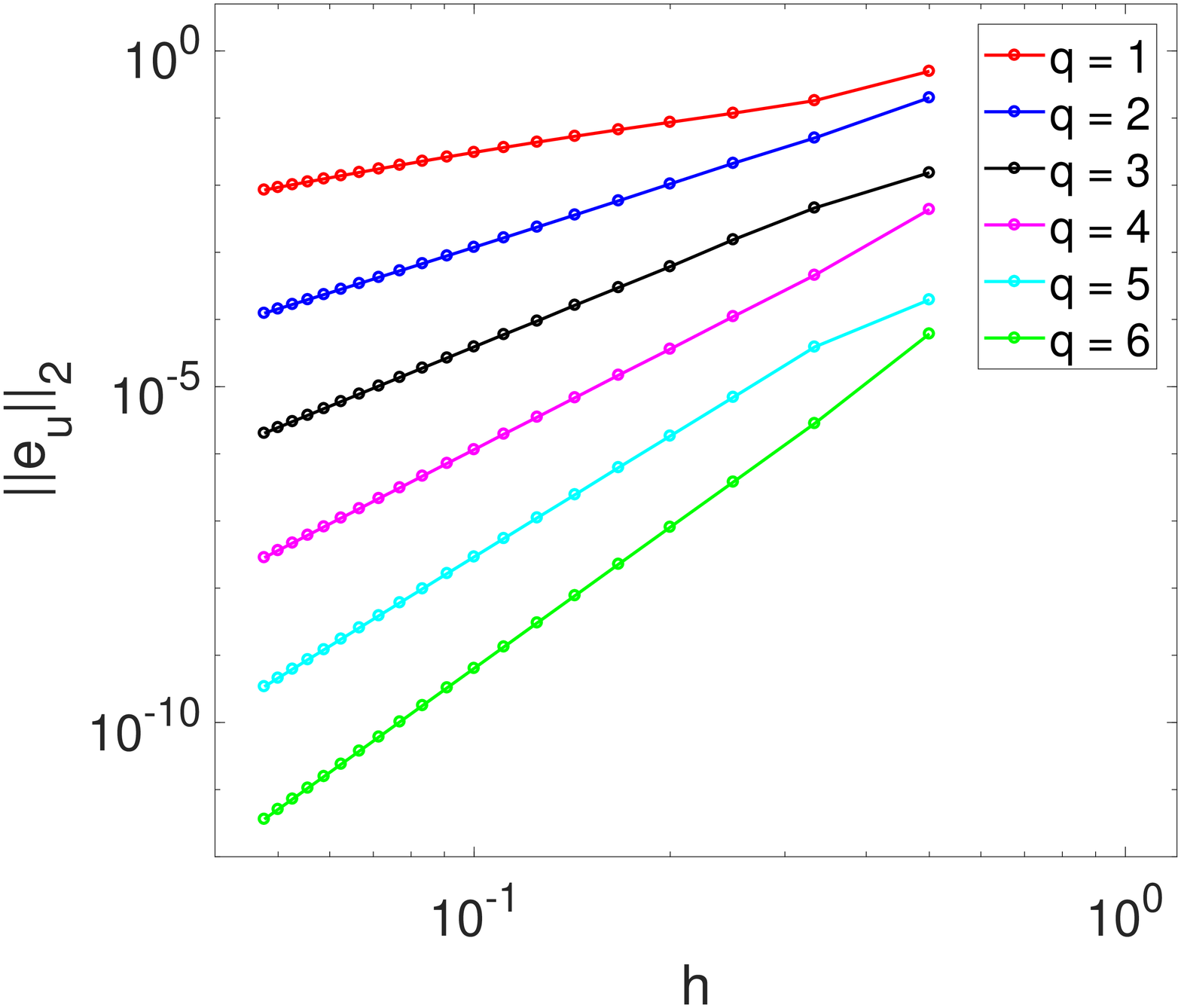}}
\subfloat[$w_{x} = 0.5, w_{y} = 1.5, c = 1$]{\label{fig:f}\includegraphics[width=0.5\textwidth]{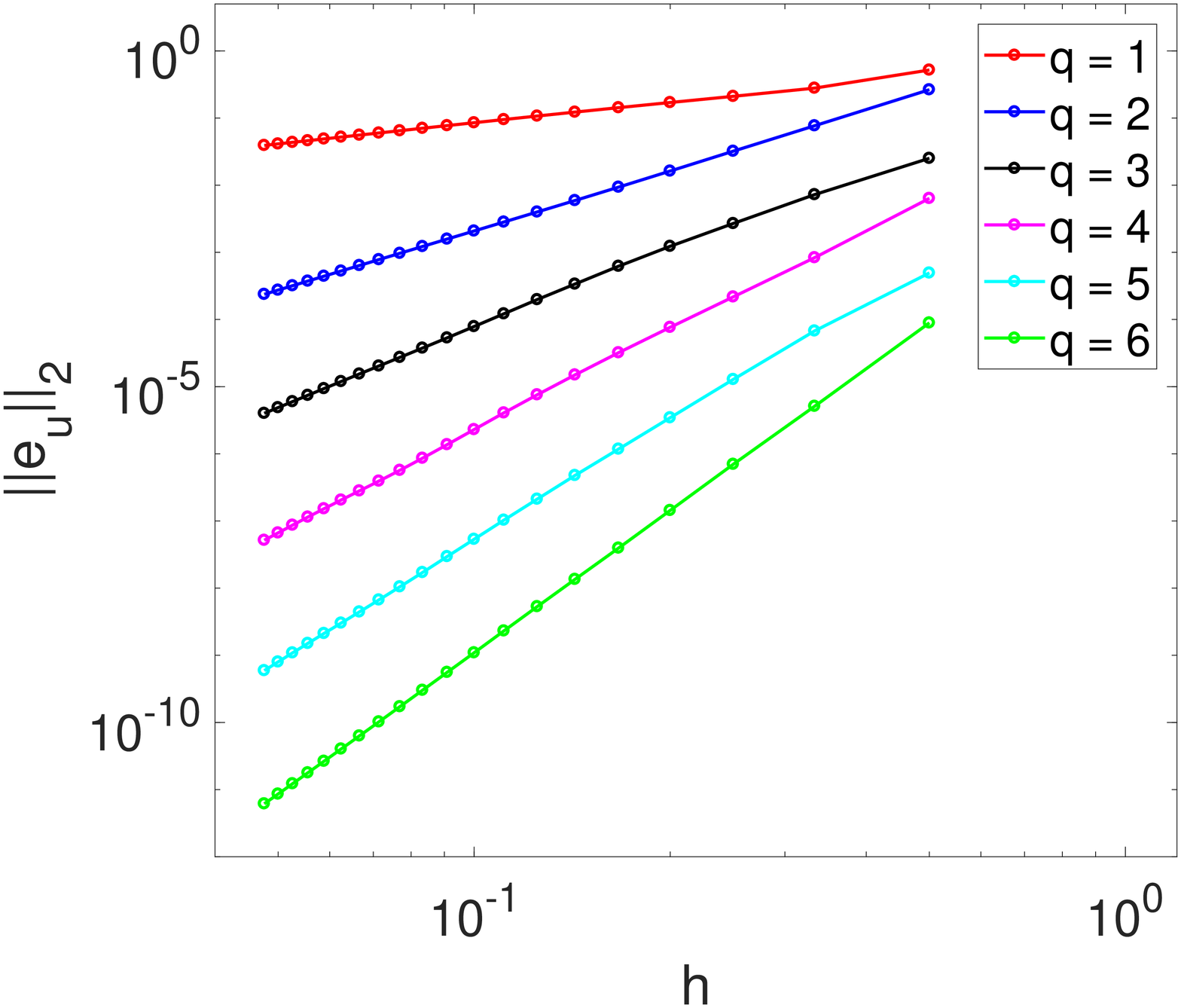}}\\
\subfloat[$w_{x} = w_{y} = 0.5, c = 1$]{\label{fig:g}\includegraphics[width=0.5\textwidth]{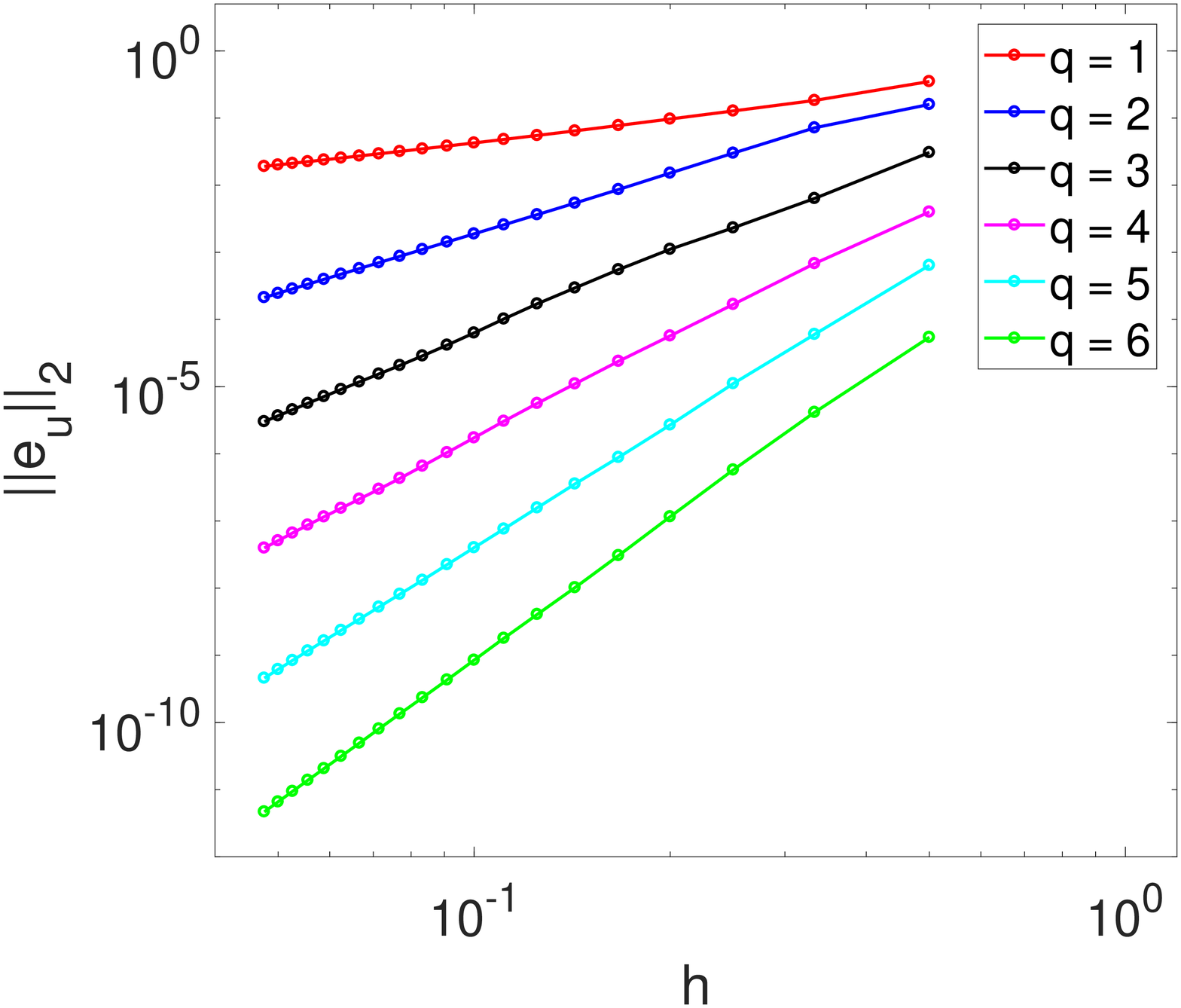}}
\end{center}
\caption{Plots of the error in $u$ as a function of $h$ in 2d with the upwind flux and periodic boundary conditions. In the legend, $q$ is the degree of the approximation of $u$ and $v$ for both $x$ and $y$ directions.\label{fig:u_ou_per}}
\end{figure}

The errors for $u$ obtained with the upwind flux are plotted against the grid-spacing $h$ in Figure \ref{fig:u_ou_per}. Linear regression estimates of the rate of convergence can be found in Table \ref{R_O_per_LR2}. We observe convergence at the optimal rate, $q+1$, for $u$ and a convergence rate of $q$ for $v$ if $q \geq 2$.

\begin{figure}[tbhp]
\begin{center}
\subfloat[$w_{x}=w_{y}=0.5, c=1$]{\label{fig:a}\includegraphics[width=0.5\textwidth]{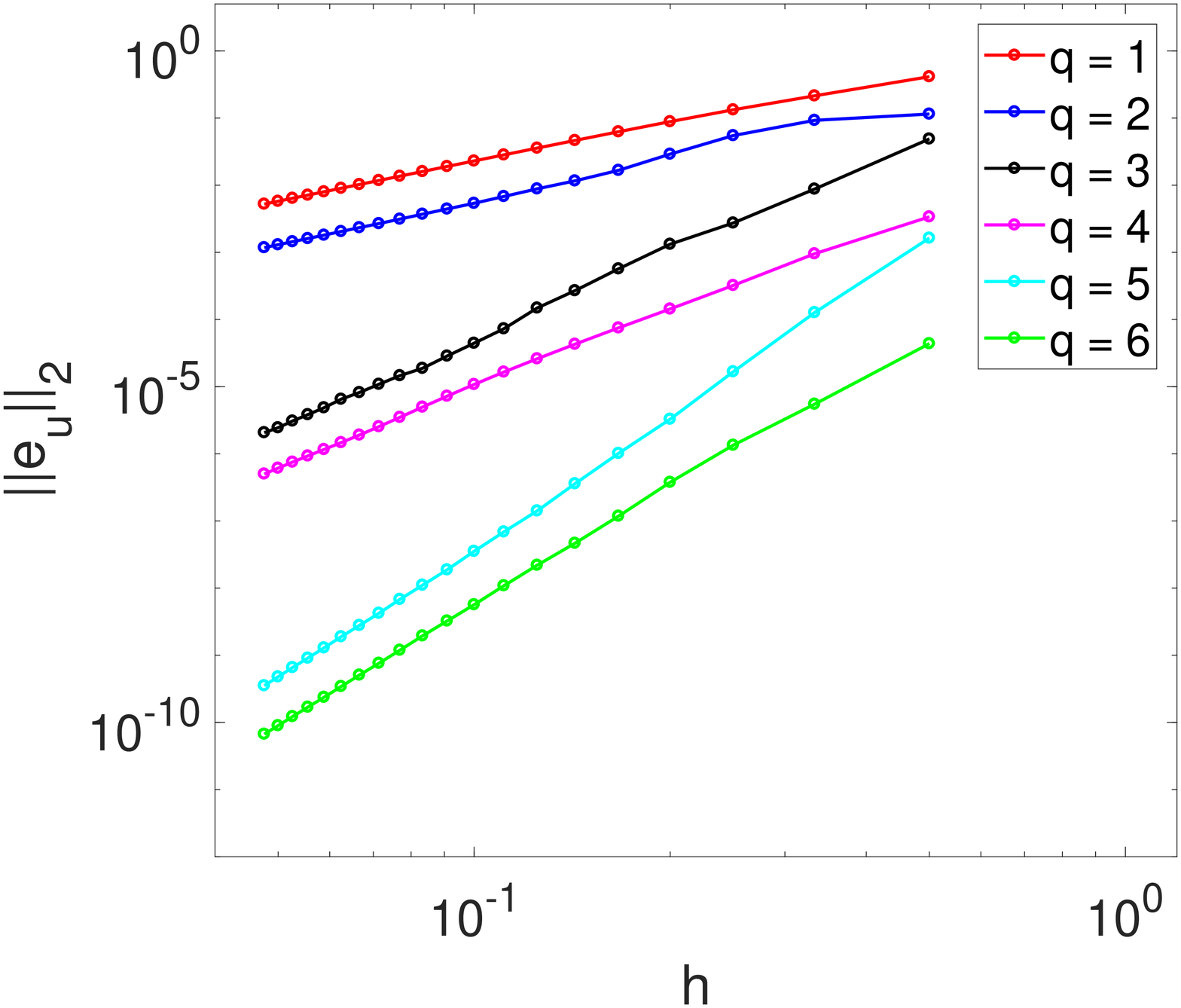}}
\subfloat[$w_{x}=w_{y}=c=1$]{\label{fig:b}\includegraphics[width=0.5\textwidth]{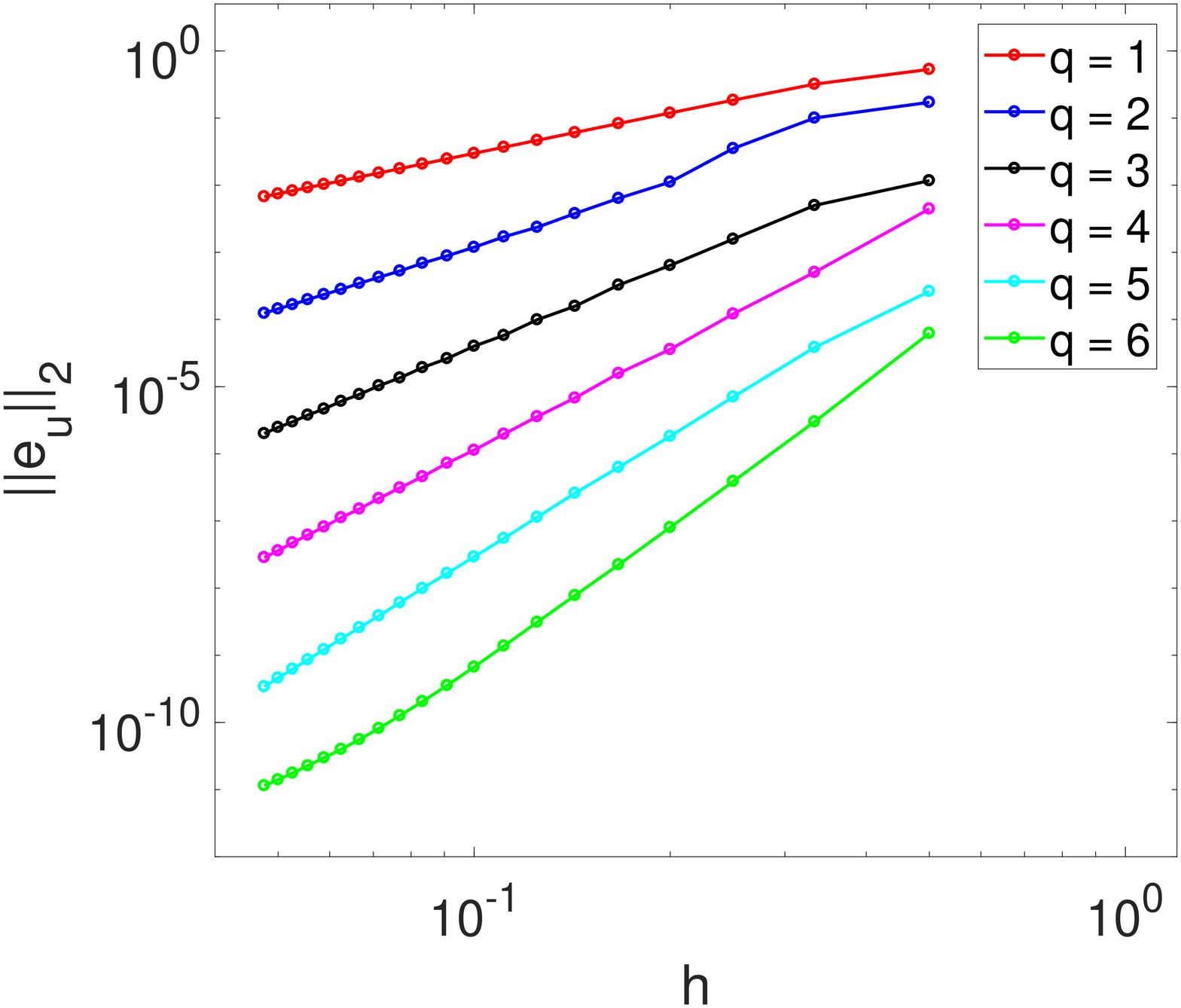}}\\
\subfloat[$w_{x}=0.5,w_{y}=1.5,c=1$]{\label{fig:c}\includegraphics[width=0.5\textwidth]{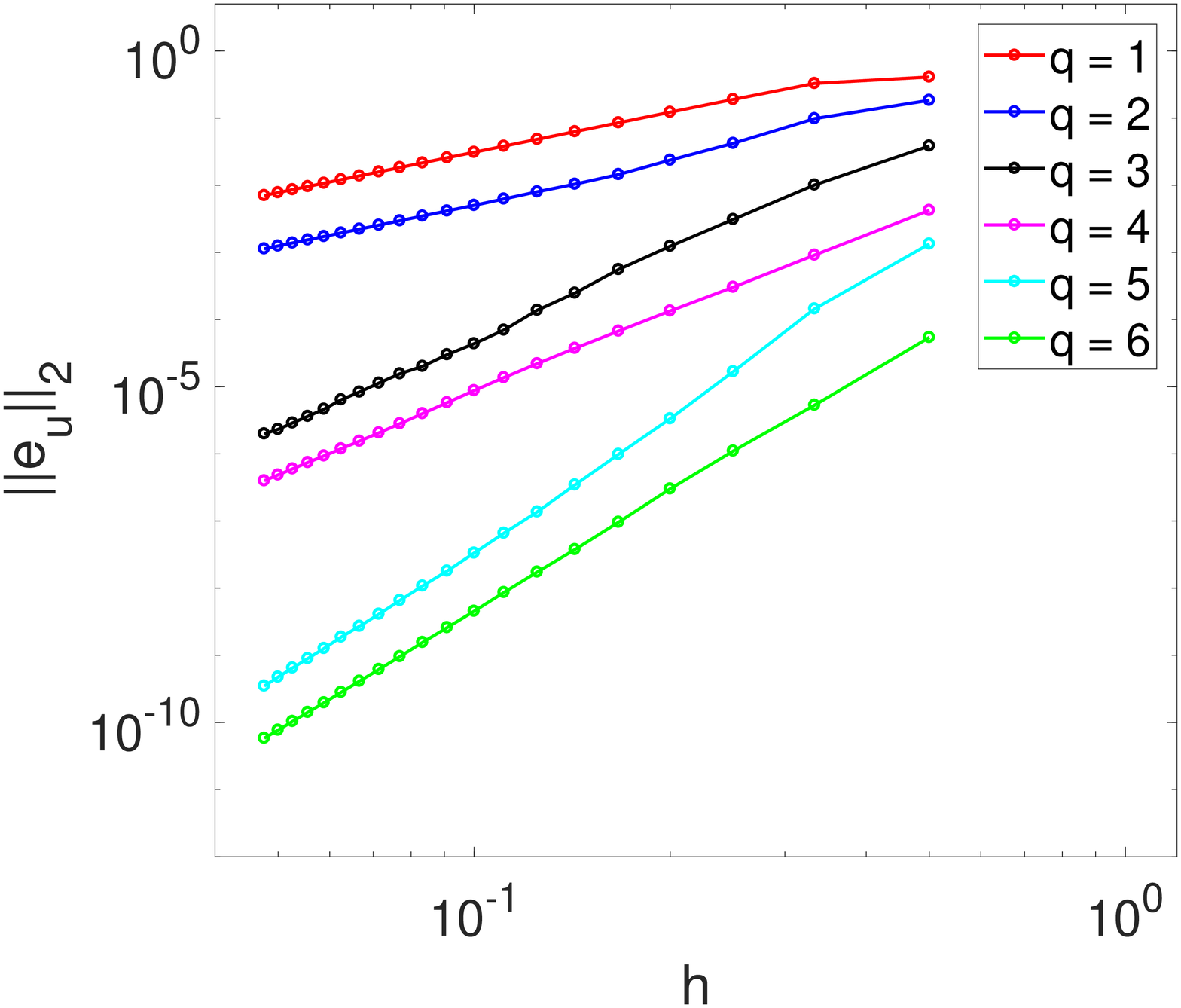}}
\subfloat[$w_{x}=w_{y}=1.5,c=1$]{\label{fig:d}\includegraphics[width=0.5\textwidth]{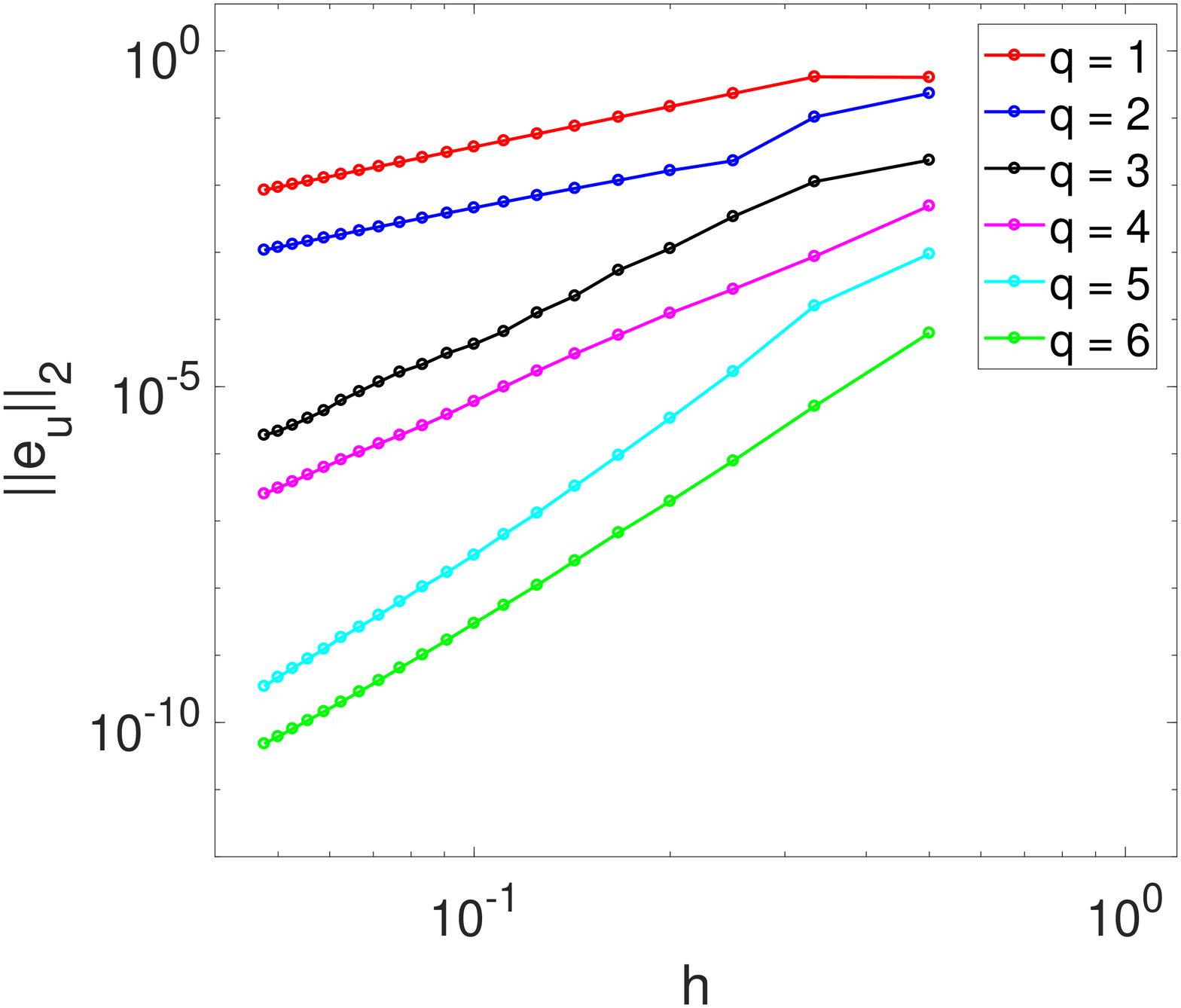}}
\caption{Plots of the error in $u$ as a function of $h$ in 2d with the central flux and periodic boundary conditions. In the legend, $q$ is the degree of approximation to $u$ and $v$ for both $x$ and $y$ directions.\label{fig:u_center_per}}
\end{center}
\end{figure}

\begin{table}
\caption{Linear regression estimates of the convergence rate of $u$ and $v$ in $2d$ with central flux for periodic boundary condition and $q_{x} = q_{y} = q$. \label{R_center_per_LR}}
\begin{center}
  \begin{tabular}{|l|c c c c c c|}
  \hline
 Degree $(q)$ of approx. of $u$ and $v$   & 1 & 2 & 3 & 4 & 5 & 6 \\
\hline
  Rate fit $u$ $(w_{x}=0.5,w_{y}=0.5,c=1)$ & 2.00 & 2.04 & 4.04& 4.06 & 6.15 & 6.01 \\
Rate fit $v$ $ (w_{x}=0.5,w_{y}=0.5,c=1)$ & 0.96 & 0.99 & 3.08& 2.97 & 5.15 & 4.99 \\
\hline
Rate fit $u$ $ (w_{x}=1,w_{y}=1,c=1)$ &2.00& 3.05 & 4.01 & 4.97 & 6.01 & 5.13 \\
Rate fit $v$ $ (w_{x}=1,w_{y}=1,c=1)$ &1.00& 2.05 & 2.95 & 3.99 & 4.96 & 6.01 \\
\hline
Rate fit $u$ $ (w_{x}=0.5,c=1,w_{y}=1.5)$&2.00 & 2.01 & 4.30& 4.09 & 6.11 & 5.86 \\
Rate fit $v$ $ (w_{x}=0.5,c=1,w_{y}=1.5)$&0.97 & 0.99 & 3.09& 2.98 & 5.07 & 4.99 \\
\hline
Rate fit $u$ $ (w_{x}=1.5,w_{y}=1.5,c=1)$ &2.00& 1.96 & 4.56 & 4.20 & 6.06 & 5.45 \\
Rate fit $v$ $ (w_{x}=1.5,w_{y}=1.5,c=1)$ &1.79& 1.02 & 3.37 & 3.23 & 4.55 & 4.97 \\
\hline
\end{tabular}
\end{center}
\end{table}

The $L^2$ error for $u$ for the central flux is plotted against the grid-spacing $h$ in Figure \ref{fig:u_center_per}. Linear regression estimates of the rate of convergence can be found in Table \ref{R_center_per_LR} for both $u$ and $v$. Similar to the one-dimensional case, convergence is optimal for $u$ when $q$ is odd and suboptimal by one when $q$ is even except in the special case of sonic boundaries.

\subsection{Dirichlet and radiation boundary conditions in two space dimensions}
Lastly we consider a problem with a Dirichlet boundary condition on inflow boundaries (left and bottom) and radiation boundary condition on outflow boundaries (right and top). Since we don't have a simple exact solution satisfying these boundary conditions, we set
\bd
u(x,y,t) = x(1-x)^{2}y(1-y)^{2}\exp(x+y)\sin(t),
\ed
and solve
\bd 
    (\frac{\partial}{\partial t} + \textbf{w}\cdot\nabla)^{2} u = c^{2}\Delta u + f, \ \ (x, y) \in (0, 1)\times(0, 1),  \ \ t > 0,
\ed
with $f$ determined by $u$. Note that for this specific choice we have that $u(x,y,t) = 0$ on the inflow boundaries and $u(x,y,t) = u_{x}(x,y,t) = u_{y}(x,y,t) = 0$ on the outflow boundaries. In the following numerical experiments we choose the same approximation spaces for $u^{h}$ and $v^{h}$, polynomial degrees $q_{x} = q_{y} = q$. We evolve the solution to $T = 0.2$ with the step size $\Delta t = {\rm CFL}h $ and ${\rm CFL} = 0.075/(2\pi)$. Here we only consider the subsonic case, $w_x=w_y=0.5$ with $c=1$, and compare both upwind and central fluxes. 

\begin{figure}[htb]
\begin{center}
\subfloat[$w_{x} = w_{y} = 0.5,\, c =1$] {\label{fig:gd}\includegraphics[width=0.5\textwidth]{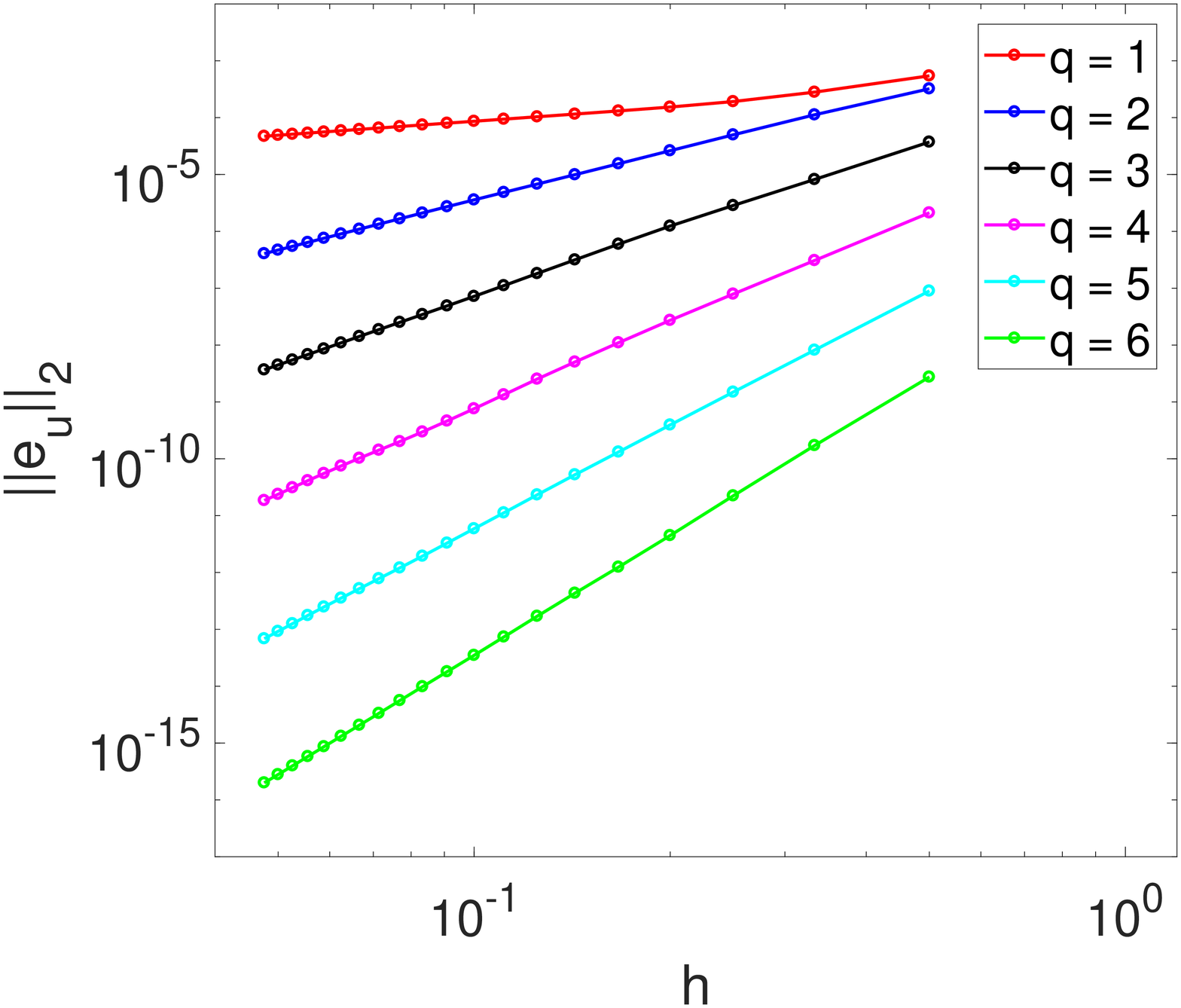}}
\subfloat[$w_{x} = w_{y} = 0.5,\, c = 1$]{\label{fig:ad}\includegraphics[width=0.5\textwidth]{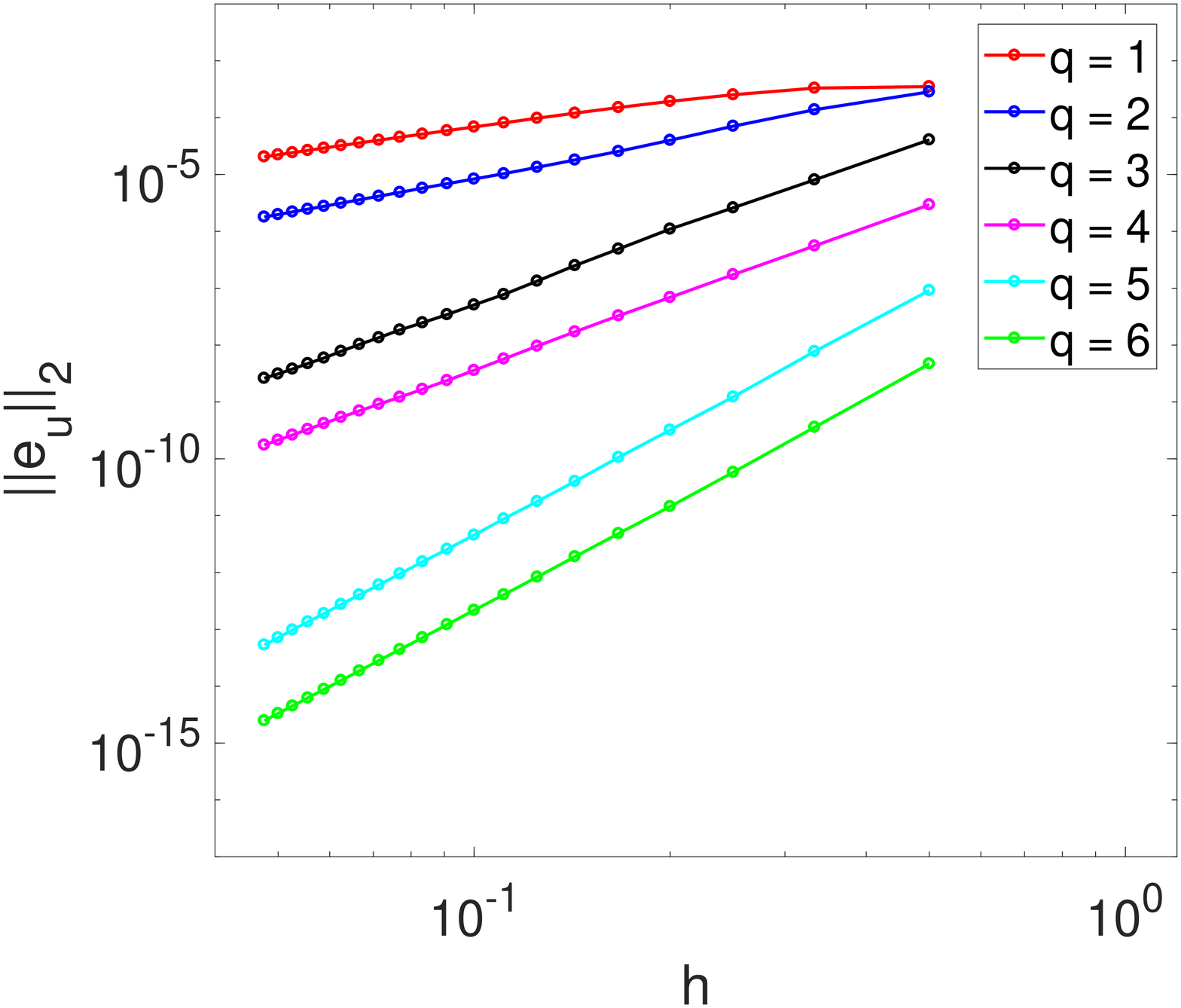}}
\end{center}
\caption{Plots of the error in $u$ as a function of $h$ in 2d with upwind (left) and central (right) fluxes for Dirichlet boundary condition on inflow boundaries and a radiation boundary condition on outflow boundaries. In the legend, $q$ is the degree of the approximation to $u$ and $v$ for both $x$ and $y$ directions.\label{fig:u_dbc}}
\end{figure}

\begin{table}
\caption{Linear regression estimates of the convergence rate for $u$ and $v$ in $2d$ with Dirichlet boundary condition on inflow boundaries, radiation boundary condition on outflow boundaries and $q_{x} = q_{y} = q$. Here the first two rows correspond to the upwind flux and the last two to the central flux. \label{DRBCtab}}
\begin{center}
  \begin{tabular}{|l|c c c c c c|}
  \hline
Degree $(q)$ of approx. of $u$  & 1 & 2 & 3 & 4 & 5 & 6 \\
\hline
Rate fit $u$ $ (w_{x}=0.5,w_{y}=0.5,c=1)$&0.82& 2.94 & 4.01 & 4.96 & 5.97 & 6.96 \\
Rate fit $v$ $ (w_{x}=0.5,w_{y}=0.5,c=1)$&0.78& 1.87 & 2.92 & 3.92 & 4.95 & 5.97 \\
\hline
Rate fit $u$ $ (w_{x}=0.5,w_{y}=0.5,c=1)$ &1.65 & 2.09 & 4.09& 4.04 & 6.01 & 6.01 \\
Rate fit $v$ $ (w_{x}=0.5,w_{y}=0.5,c=1)$ &0.93 & 0.98 & 2.98& 3.00 & 5.01 & 5.00 \\
\hline
\end{tabular}
\end{center}
\end{table}

The error for $u$ is plotted against the grid-spacing $h$ for both fluxes in Figure \ref{fig:u_dbc}. Linear regression estimates of the rate of convergence can be found in Table \ref{DRBCtab}. The rates of convergence are very close to those for the periodic problem.

\section{Conclusion and extension}
In conclusion, we have generalized the energy-based discontinuous Galerkin method of \ci{DGwave} to the wave
equation with advection, a problem for which the energy density takes a more complicated form than a
simple sum of a term involving the time derivative and a term involving space derivatives. We have
shown that the new form can be handled by introducing a second variable which, unlike what was done
in \ci{DGwave,el_dg_dath}, involves both space and time derivatives. We prove error estimates
completely analogous with those shown in \ci{DGwave} for the isotropic wave equation, including cases with
both subsonic and supersonic background flows. Numerical experiments also demonstrate optimal convergence
on regular grids when an upwind flux is used.

A potential application of the method would be to linearized models in aeroacoustics, where its generalization
to inhomogeneous media such as those defined by background shear flows would be needed
(e.g. \ci{Goldstein03}). Here we expect that
the use of upwind fluxes would guarantee stability for the discretization of the principal part which should
be sufficient to establish convergence. Secondly, we will understand our construction in the context of
regularly hyperbolic systems as defined in \ci[Ch. 5]{CAction} with the hope of treating the general case.

\bibliography{hag}
\bibliographystyle{siam}

\end{document}